\documentclass[journal,twocolumn, 10pt]{IEEEtran}
\usepackage{setspace}
\usepackage{color}
\usepackage{cite}
\usepackage{graphicx}
\usepackage{subfigure}
\usepackage{amsmath}
\usepackage{amssymb}
\usepackage{amsthm}
\usepackage{amsmath}
\usepackage{algorithm,algorithmic}
\usepackage{booktabs}
\usepackage{multirow}

\definecolor{darkblue}{rgb}{0,0,0.7}

\newcommand{\ip}[2]{\left\langle {#1}, {#2} \right\rangle}

\newcommand{\bdsb}[1]{\boldsymbol{#1}}

\newtheorem{lem}{Lemma}

\newtheorem{corollary}{Corollary}
\newtheorem{thm}{Theorem}

\newtheorem{assumption}{Assumption}

\title{Convergence and Applications of a Gossip-based Gauss-Newton Algorithm}

\author{Xiao Li, {\it Student Member, IEEE},  and Anna Scaglione, {\it Fellow, IEEE}
\thanks{This work was supported by the TCIPG project sponsored by Department of Energy under the Award DE-OE0000097.}
\thanks{%
  The authors are with the
  Department of Electrical and Computer Engineering,
  University of California, Davis,
  One Shields Avenue, Kemper Hall,
  Davis, California 95616-5294
  (email : \{eceli,ascaglione\}@ucdavis.edu).}
}

\begin{document}

\maketitle

\graphicspath{{./figure/}}

\begin{abstract}
The Gauss-Newton algorithm is a popular and efficient centralized method for solving non-linear least squares problems. In this paper, we propose a multi-agent distributed version of this algorithm, named Gossip-based Gauss-Newton (GGN) algorithm, which can be applied in general problems with non-convex objectives. Furthermore, we analyze and present sufficient conditions for its convergence and show numerically that the GGN algorithm achieves performance comparable to the centralized algorithm, with graceful degradation in case of network failures. More importantly, the GGN algorithm provides significant performance gains compared to other distributed first order methods.
\end{abstract}


\vspace{-0.6cm}
\section{Introduction}

Numerical algorithms for solving non-linear least squares (NLLS) problems are well studied and understood \cite{nocedal1999numerical}. Popular methods are the so called {\it Newton} and {\it Gauss-Newton} algorithms.  Newton algorithms are second order methods that use the Hessian of the objective function to stabilize and accelerate local convergence \cite{dennis1996numerical,boyd2004convex}, while Gauss-Newton simplifies the computation of the Hessian particularly for NLLS problems by ignoring the higher order derivatives \cite{bjorck1996numerical}. The Gauss-Newton algorithm is commonly used for power systems state estimation \cite{monticelli2000electric}, localization \cite{mensing2006positioning}, frequency estimation \cite{stoica1989maximum}, Kalman filtering \cite{bell1993iterated}, medical imaging \cite{schweiger2005gauss}. Given the fact that for some of these problems the data are acquired over a wide area, in this paper we are interested in the decentralized implementation of the Gauss-Newton algorithm in a network, via {\it gossiping}.
%
%
Since their introduction \cite{tsitsiklis1984problems}, gossip algorithms have been extensively investigated \cite{karp2000randomized,olfati2004consensus}, as surveyed in \cite{dimakis2010gossip}. Deterministic and randomized protocols for gossip algorithms with synchronous or asynchronous updates have been further studied \cite{kempe2003gossip,boyd2006randomized} and applied in different areas in networked control and distributed signal processing, such as distributed Kalman filtering \cite{olfati2007consensus} or convex optimization problems \cite{johansson2008subgradient}.

Our work is closely related with the recent developments in the area of distributed optimization via {\it network diffusion}, which evolved from the incremental methods in \cite{bertsekas1997new,nedic2001incremental} and gossip-based sub-gradient algorithms in \cite{johansson2008subgradient} onto fully decentralized and randomized algorithms. The distributed algorithms analyzed in \cite{nedic2009distributed,ram2010distributed,nedich2010asynchronous,srivastava2011distributed,kar2008distributed} tackle convex optimization problems through either synchronous or asynchronous communications. These techniques combine a local descent step with a network diffusion step. The convergence of these diffusion algorithms typically requires {\it convexity} and a diminishing step-size, which results in slow convergence in general \cite{matei2011performance}. Recently, \cite{chen2011diffusion} assumes {\it local strong convexity} and proposes a diffusion optimization scheme for general convex problems by using stochastic gradients with a constant step-size. Furthermore, the convergence analysis of network diffusion algorithms has also been developed for adaptive formulations using a constant step-size for linear filtering problems \cite{lopes2008diffusion,cattivelli2010diffusion,cattivelli2008diffusion}, or using a diminishing step-size for non-linear invertible systems \cite{kar2008distributed}. Despite the simplicity of first order methods in diffusion algorithms, they generally suffer from slow convergence in contrast to Newton-type algorithms.

Recently, a gossip-based Newton method was derived in \cite{wei2010distributed} to solve network utility maximization problems and later applied to power flow estimation \cite{ilic2012toward}. The algorithm relies on the diagonal structure of the Hessian matrix and its convergence  is proven under the hypothesis that the error of the computed Newton descent is bounded.  In addition, the method is developed specifically for {\it strictly convex} problems, where the variables are completely separable for each distributed agent (i.e., its Hessian is block diagonal), while NLLS problems are oftentimes non-convex and non-separable. Although there have been some ad-hoc applications of the Gauss-Newton methods via network average consensus in sensor networks \cite{bejar2010distributed,cheng2005distributed,calafiore2010distributed} or incremental methods in acoustic sources localization \cite{zhao2007information} that relax these assumptions, a thorough study of the algorithm performance in the general case is still missing.


Motivated by this background, in this paper, we propose and study the performance of the Gossip-based Gauss-Newton (GGN) algorithm, for general NLLS problems that are non-separable and non-convex.
We also showcase its performance in power system state estimation (PSSE) \cite{schweppe1974static,larson1970state} for system monitoring and control. Recently, the development of distributed PSSE schemes has received considerable attention \cite{brice1982multiprocessor,kurzyn1983real,yang2011transition,gomez2011multilevel,falcao1995parallel,lin1992distributed,ebrahimian2000state,van1981two,zhao2005multi,jiang2007distributed} to achieve wide area awareness in the expanding power grid. Most of these algorithms hierarchically aggregate the information from distributed control areas under the assumption that there are redundant measurements available at each area to uniquely identify the local state variables (i.e., local observability). Such condition is not required by the GGN algorithm in this paper, similar to the recent works in \cite{xie2012fully,kekatos2012distributed}. In comparison, the proposed GGN algorithm is very different in terms of the network communications and algorithm convergence. The method in \cite{xie2012fully} is motivated by the diffusion algorithm in \cite{kar2008distributed} (similar to \cite{nedic2009distributed} in an adaptive setting), which is a first order sub-gradient method. On the other hand, our approach converges much faster and our communication model is more flexible and robust. The authors in \cite{kekatos2012distributed} used the {\it Alternating Direction Method of Multipliers} (ADMM) to distribute the state estimation procedure by decomposing the state variables in different areas so that each agent estimates a local state. This is in contrast to the global state considered in this paper. Furthermore, the communications entailed by ADMM is constrained by the power grid topology, while the communication model considered in this paper is decoupled from the grid topology and more flexible in terms of network reconfigurations and random failures. Also, the numerical tests in \cite{kekatos2012distributed} are based exclusively on a linear model using Phasor Measurement Unit (PMU) data, while the algorithm convergence in general is not discussed.

The challenge associated with PSSE  is the presence of multiple stationary points due to the non-convexity of the NLLS objective. This fact confirms the importance of deriving the sufficient conditions for the convergence of the GGN, provided in this paper. These conditions indicate how close the algorithm needs to be initialized around the global minimizer in order to converge to it. The criterion has practical implications in the power grid application, since it can be met by deploying judiciously PMUs (see \cite{li2013optimal}). In the simulations, we show how our GGN algorithm performs compared to the PSSE diffusion algorithms in \cite{xie2012fully} and \cite{kar2008distributed} in an adaptive setting with streaming data.

{\bf Synopsis:} In Section \ref{problem_statement}, we define the NLLS problems and provide the distributed NLLS formulation in a network. Then, the proposed GGN algorithm is described in detail in Section \ref{GGN} and its convergence analysis follows in Section \ref{convergence_analysis}. We formulate the PSSE application in Section \ref{PSSE} as a NLLS problem and solve it using the proposed GGN algorithm. Finally, the convergence and performance of the GGN algorithm is demonstrated for PSSE problems in Section \ref{numerical_results}.

{\bf Notation:} We denote vectors (matrices) by boldface lower-case (upper-case) symbols, and the set of real (complex) numbers by $\mathbb{R}$ ($\mathbb{C}$). The magnitude of a complex number $x$ is denoted by $|x|=\sqrt{xx^\ast}$, where $x^\ast$ is the conjugate of $x$. The transpose, conjugate transpose, and inverse of a non-singular matrix $\mathbf{X}$ are denoted by $\mathbf{X}^T$, $\mathbf{X}^H$ and $\mathbf{X}^{-1}$, respectively. The inner product between two vectors $\mathbf{x},\mathbf{y}\in\mathbb{C}^{N\times 1}$ is defined accordingly as $\ip{\mathbf{x}}{\mathbf{y}} = \sum_{n=1}^N y_n^\ast x_n$. The $\mathbf{W}$-weighted  Euclidean norm of a vector $\mathbf{x}$ is denoted by $\left\|\mathbf{x}\right\|_{\mathbf{W}}=\sqrt{\mathbf{x}^H\mathbf{W}\mathbf{x}}$, and the conventional Euclidean norm is written as $\left\|\mathbf{x}\right\|$. The $2$-norm of a matrix $\mathbf{A}$ is denoted by $\|\mathbf{A}\|$ and the {\it Frobenius} norm is denoted by $\|\mathbf{A}\|_F$. Given a matrix $\mathbf{A}=[\mathbf{a}_1,\cdots,\mathbf{a}_N]$ where $\mathbf{a}_n$ is a column vector, the vectorization operator is defined as $\mathrm{vec}(\mathbf{A})=[\mathbf{a}_1^T,\cdots,\mathbf{a}_N^T]^T$.

\section{Problem Statement}\label{problem_statement}
Let $\mathbf{x}\in\mathbb{R}^N$ be an unknown parameter vector associated with a specific network, belonging to a compact convex set $\mathbb{X}$. The network objective is described by a vector-valued continuously differentiable function $\mathbf{g}(\mathbf{x})=[g_1(\mathbf{x}),\cdots,g_M(\mathbf{x})]^T$ with $M$ outputs, defined as $g_m : \mathbb{R}^N \rightarrow \mathbb{R}$, $m=1,\cdots,M$. Note that $\{g_m\}_{m=1}^M$ are not necessarily convex. Then, a non-linear least squares (NLLS) problem for the network is
\begin{align}\label{centralized_NLLS}
	\underset{\mathbf{x}\in\mathbb{X}}{\min}~~ \|\mathbf{g}(\mathbf{x})\|^2.
\end{align}
Throughout this paper, we assume the following about \eqref{centralized_NLLS}:

\begin{assumption}\label{lipshitz}
~
\begin{enumerate}
    \item The vector function is continuous, differentiable, and bounded for $\mathbf{x}\in\mathbb{X}$ with
        \begin{align}\label{epsilon_max}
        \|\mathbf{g}(\mathbf{x})\| \leq  \epsilon_{\max} .
        \end{align}
    \item The $M\times N$ Jacobian $\mathbf{G}(\mathbf{x})={\partial \mathbf{g}(\mathbf{x})}/{\partial \mathbf{x}^T}$ is full-column rank for all $\mathbf{x}\in\mathbb{X}$. Denote by $\lambda_{\min}(\cdot)$ and $\lambda_{\max}(\cdot)$ the minimum and maximum eigenvalues and let
		\begin{align*}
			\sigma_{\min} &= \underset{\mathbf{x}\in\mathbb{X}}{\min}~\sqrt{\lambda_{\min}\left(\mathbf{G}^T(\mathbf{x})\mathbf{G}(\mathbf{x})\right)},\\
			\sigma_{\max} &= \underset{\mathbf{x}\in\mathbb{X}}{\max}~\sqrt{\lambda_{\max}\left(\mathbf{G}^T(\mathbf{x})\mathbf{G}(\mathbf{x})\right)},
		\end{align*}
	    with $0<\sigma_{\min}\leq\sigma_{\max}<\infty$.
    \item The Jacobian $\mathbf{G}(\mathbf{x})$ satisfies the Lipschitz condition
		\begin{align*}
			\left\|\mathbf{G}(\mathbf{x})-\mathbf{G}(\mathbf{x}')\right\|
            &\leq \omega \left\|\mathbf{x}-\mathbf{x}'\right\|,\quad \mathbf{x},\mathbf{x}'\in\mathbb{X},
		\end{align*}
		where $\omega>0$ is a Lipschitz constant.
\end{enumerate}
\end{assumption}

\subsection{Centralized Gauss-Newton Algorithm}
When data and functions are available at a central point, the Gauss-Newton method starts from some initial point $\mathbf{x}^0$ and solves the NLLS problem iteratively \cite{bjorck1996numerical}
\begin{align}\label{algorithm_update}
	\mathbf{x}^{k+1} = P_{\mathbb{X}}\left[\mathbf{x}^k-\alpha_k\mathbf{d}^k\right], \quad k=1, 2, \cdots,
\end{align}
where $\alpha_k$ is the step-size in the $k$-th iteration and $P_{\mathbb{X}}[\cdot]$ is a projection onto the constrained set $\mathbb{X}$. According to Assumption \ref{lipshitz}, the Gauss-Newton Hessian matrix $\mathbf{G}^T(\mathbf{x}^k)\mathbf{G}(\mathbf{x}^k)$ is positive definite, hence the resulting $\mathbf{d}^k$ constitutes a descent direction of the objective function
\begin{align}\label{central_descent}
	\mathbf{d}^k
	&=\left[\mathbf{G}^T(\mathbf{x}^k)\mathbf{G}(\mathbf{x}^k)\right]^{-1}
	    \mathbf{G}^T(\mathbf{x}^k)\mathbf{g}(\mathbf{x}^k),
\end{align}
where $\mathbf{G}(\mathbf{x})$ is the $M\times N$ Jacobian matrix of $\mathbf{g}(\mathbf{x})$. In this paper, we assume that fixed points always exist for the update \eqref{algorithm_update}, which corresponds to the set of the stationary points of the cost function satisfying the first order condition
\begin{align}\label{fixed_joint}
	\mathbf{G}^T (\mathbf{x}^\star)\mathbf{g}(\mathbf{x}^\star) = \mathbf{0},\quad \mathbf{x}^\star\in\mathbb{X}.
\end{align}
Note that if $\alpha_k$ is chosen differently at each iteration, the algorithm is called the {\it damped Gauss-Newton} method while $\alpha_k=\alpha$ corresponds to the {\it undamped Gauss-Newton} method. Under Assumption \ref{lipshitz}, it is well-known from \cite{nocedal1999numerical,bjorck1996numerical} that if the step-size $\alpha_k$ is chosen according to the Wolfe condition, the Gauss-Newton iteration converges to a stationary point of the cost function. Since many NLLS problems are non-convex by nature, the focus in this paper is to study the local convergence property of the algorithm to an arbitrary fixed point $\mathbf{x}^\star\in\mathbb{X}$.

\subsection{Distributed Formulation}
Although the convergence of centralized Gauss-Newton algorithms is well studied \cite{bjorck1996numerical} under Assumption \ref{lipshitz}, it is not immediately clear that similar local convergence properties can be maintained for the decentralized version. As shown in Fig. \ref{fig.proposed_architecture}, suppose there are $I$ distributed {\it agents}, and the $i$-th {\it agent} only knows a subset function $\mathbf{g}_i: \mathbb{R}^N\rightarrow \mathbb{R}^{M_i}$ from \eqref{centralized_NLLS}, i.e.
\begin{align}
	\mathbf{g}(\mathbf{x}) =[\mathbf{g}_1^T(\mathbf{x}),\ldots,\mathbf{g}_I^T(\mathbf{x})	]^T
\end{align}
with $M=\sum_{i=1}^I M_i$. In this setting, the goal is to obtain
\begin{align}\label{minimizer}
	\widehat{\mathbf{x}} = \arg\underset{\mathbf{x}\in\mathbb{X}}{\min} ~ \sum_{i=1}^I \left\|\mathbf{g}_i(\mathbf{x})\right\|^2,
\end{align}
where each agent has only partial knowledge of the global cost function. Based on Assumption \ref{lipshitz}, we have the following results on the distributed formulation.

\begin{corollary}\label{ass_lipschitz}
Let Assumption \ref{lipshitz} hold. Given that the partial Jacobian $\mathbf{G}_i(\mathbf{x})=\partial \mathbf{g}_i(\mathbf{x})/\partial\mathbf{x}^T$ is a sub-matrix of the full Jacobian $\mathbf{G}(\mathbf{x})$, then we have (cf. \cite[Corollary 3.1.3]{horntopics})
\begin{align*}
	\left\|\mathbf{G}_i(\mathbf{x})-\mathbf{G}_i(\mathbf{x}')\right\|
       &\leq \omega \left\|\mathbf{x}-\mathbf{x}'\right\|,\quad \mathbf{x},\mathbf{x}'\in\mathbb{X}.
\end{align*}
and furthermore the following conditions (cf. \cite[Theorem 12.4]{eriksson2004applied}) for arbitrary $\mathbf{x},\mathbf{x}'\in\mathbb{X}$
\begin{align*}
    \left\|\mathbf{G}_i^T(\mathbf{x})\mathbf{g}_i(\mathbf{x})-\mathbf{G}_i^T(\mathbf{x}')\mathbf{g}_i(\mathbf{x}')\right\|    &\leq \nu_{\delta}\left\|\mathbf{x}-\mathbf{x}'\right\|\\
    \left\|\mathbf{G}_i^T(\mathbf{x})\mathbf{G}_i(\mathbf{x})-\mathbf{G}_i^T(\mathbf{x}')\mathbf{G}_i(\mathbf{x}')\right\|    &\leq \nu_{\Delta}\left\|\mathbf{x}-\mathbf{x}'\right\|,
\end{align*}
where $\nu_{\delta}\geq \omega ( \epsilon_{\max}  + \sigma_{\max})$ and $\nu_{\Delta}\geq 2 \sigma_{\max} \omega$ are the associated Lipschitz constants.
\end{corollary}

\begin{figure}[t]
\centering
\includegraphics[width=0.9\linewidth]{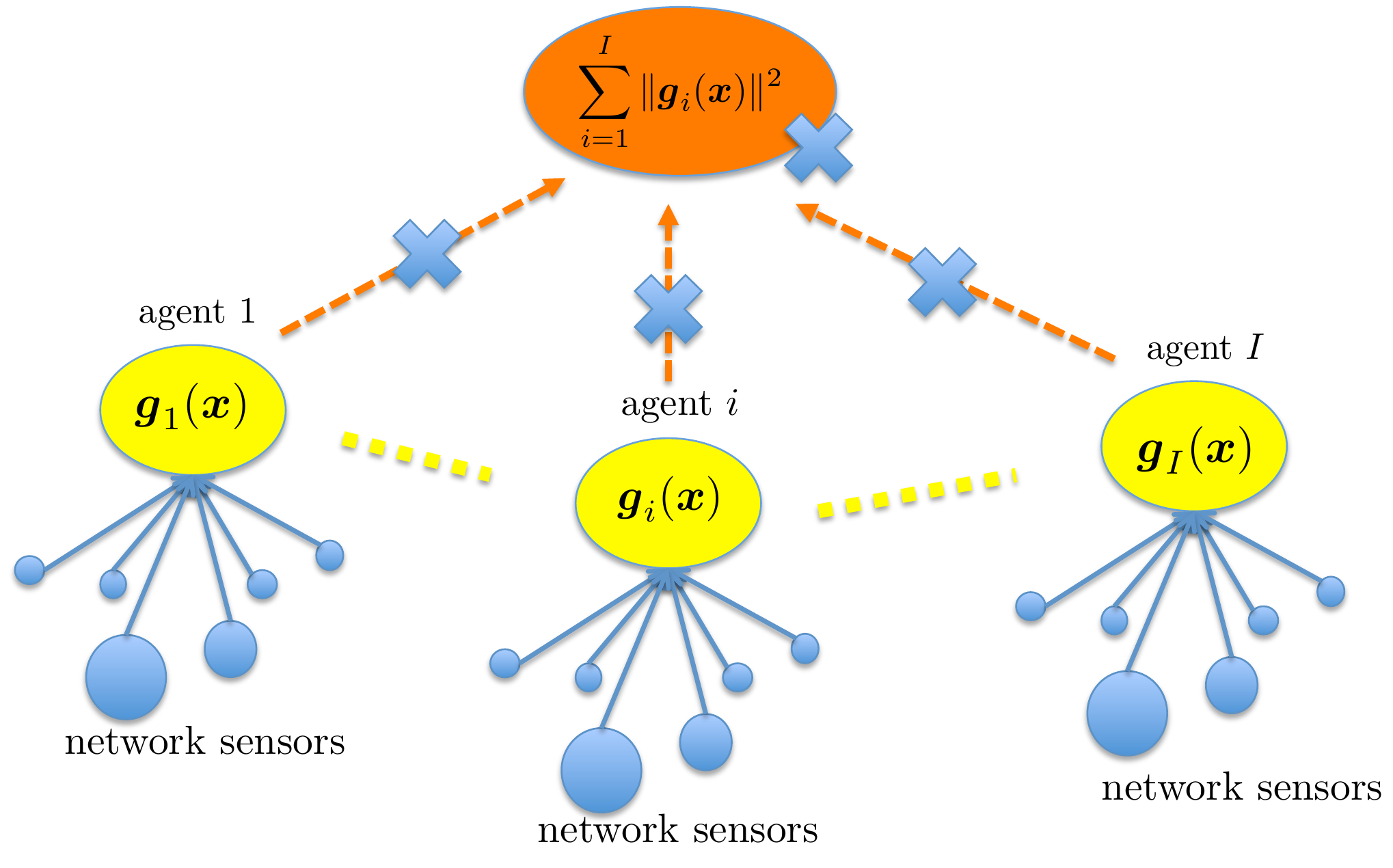}
\vspace{-0.2cm}
\caption{Schematic of multi-agent computation structure.}\label{fig.proposed_architecture}
\vspace{-0.4cm}
\end{figure}

In the distributed setting, it is difficult to coordinate the step-size at different agents to satisfy the Wolfe condition \cite{nocedal1999numerical} in a global sense. A variable step-size is also quite inconvenient, because of the difficulties of coordinating a change in the step-size across a network. As a result, we study the {\it undamped Gauss-Newton} case with a constant step-size $\alpha\in (0,1]$ i.e.
\begin{align}\label{undamped_GN}
	\mathbf{x}_i^{k+1} = P_{\mathbb{X}}\left[\mathbf{x}_i^k-\alpha\mathbf{d}_i^k\right],
\end{align}
where the exact decentralized descent is given by
\begin{align}\label{decentralized_descent}
	\mathbf{d}_i^k
	&=\left[\mathbf{G}^T(\mathbf{x}_i^k)\mathbf{G}(\mathbf{x}_i^k)\right]^{-1}
	    \mathbf{G}^T(\mathbf{x}_i^k)\mathbf{g}(\mathbf{x}_i^k).
\end{align}

According to \eqref{decentralized_descent}, each agent requires the computation of
\begin{align}
 	\mathbf{G}^T (\mathbf{x}_i^k) \mathbf{G}(\mathbf{x}_i^k) &= \sum_{j=1}^{I}\mathbf{G}_j^T(\mathbf{x}_i^k) \mathbf{G}_j(\mathbf{x}_i^k)\label{sum_Hessian}\\
	\mathbf{G}^T(\mathbf{x}_i^k)\mathbf{g}(\mathbf{x}_i^k) &= \sum_{j=1}^{I}\mathbf{G}_j^T(\mathbf{x}_i^k)\mathbf{g}_j(\mathbf{x}_i^k)\label{sum_grad},
\end{align}
while the $i$-th agent has only partial information available to compute $\mathbf{G}_i^T(\mathbf{x}_i^k) \mathbf{G}_i(\mathbf{x}_i^k)$ and $\mathbf{G}_i^T(\mathbf{x}_i^k)\mathbf{g}_i(\mathbf{x}_i^k)$. In the next section, we introduce the GGN algorithm.

\section{Gossip-based Gauss-Newton (GGN) Algorithm}\label{GGN}
The proposed GGN algorithm implements the update in \eqref{central_descent} in a fully distributed manner. There are two time scales in the GGN algorithm, one is the time for Gauss-Newton {\it update} and the other is the gossip {\it exchange} between every two Gauss-Newton updates. Throughout the rest of the paper, we consistently use {\it update} (denoted by ``$k$") for the Gauss-Newton algorithm and {\it exchange} (denoted by ``$\ell$") for network gossiping. We assume that all the network agents have a synchronous clock that determines the time instants $\tau_k$ for the $k$-th algorithm update across the network. Between two updates $[\tau_k,\tau_{k+1})$, the agents exchange information via network gossiping at time $\tau_{k,\ell}\in[\tau_k,\tau_{k+1})$ for $\ell=1,\cdots,\ell_k$. 


Next, we describe the local update model for the GGN algorithm at each distributed agent in Section \ref{update_model}, and introduce in Section \ref{gossip_model} the gossip model for every exchange $\ell=1,\cdots,\ell_k$ that takes place between every two updates.

\subsection{Local Update Model}\label{update_model}
Let $\mathbf{x}_i^k$ be the local iterate at the $i$-th agent after the $k$-th update. For convenience, let
\begin{align}\label{centralized_information}
	\mathbf{q}(\mathbf{x}_i^k) &= \frac{1}{I}\sum_{j=1}^{I}\mathbf{G}_j^T(\mathbf{x}_i^k)\mathbf{g}_j(\mathbf{x}_i^k),\\
	\mathbf{Q}(\mathbf{x}_i^k) &= \frac{1}{I}\sum_{j=1}^{I}\mathbf{G}_j^T(\mathbf{x}_i^k)\mathbf{G}_j(\mathbf{x}_i^k).
\end{align}

The ``exact descent" in \eqref{decentralized_descent}, if it were to be computed at the $i$-th agent for the $(k+1)$-th update, would be
\begin{align}\label{exact_local_descent}
	\mathbf{d}_i^k
	= \mathbf{Q}^{-1}(\mathbf{x}_i^k)\mathbf{q}(\mathbf{x}_i^k),
\end{align}
which is impossible to obtain in a distributed setting. This is because of the fact that agent $j$ is not aware of the iterate $\mathbf{x}_i^k$ at other agents $i\neq j$ as well as that each node only knows its own mapping $\mathbf{g}_j$ and $\mathbf{G}_j$. In fact, the available information at the $i$-th agent after the $k$-th Gauss-Newton update is $\mathbf{G}_i^T(\mathbf{x}_i^k)\mathbf{g}_i(\mathbf{x}_i^k)$ and $\mathbf{G}_i^T(\mathbf{x}_i^k)\mathbf{G}_i(\mathbf{x}_i^k)$. Therefore, we propose to use an average surrogate for $\mathbf{q}(\mathbf{x}_i^k)$ and $\mathbf{Q}(\mathbf{x}_i^k)$
\begin{align}\label{H_h_avg}
	\bar{\mathbf{h}}_k &=  \frac{1}{I}\sum_{i=1}^{I}\mathbf{G}_i^T(\mathbf{x}_i^k)\mathbf{g}_i(\mathbf{x}_i^k),\\
	\bar{\mathbf{H}}_k &=  \frac{1}{I}\sum_{i=1}^{I}\mathbf{G}_i^T(\mathbf{x}_i^k)\mathbf{G}_i(\mathbf{x}_i^k),
\end{align}
which can be obtained via network gossiping.

After the $k$-th update by each agent at $\tau_k$, the network enters gossip exchange stage $[\tau_k,\tau_{k+1})$ to compute the surrogate $\bar{\mathbf{h}}_k$ and $\bar{\mathbf{H}}_k$. Define the length-$N_{\mathcal{H}}$ local information vector (i.e., $N_{\mathcal{H}}=N(N+1)$) at the $i$-th agent for the $\ell$-th gossip exchange
\begin{align}\label{information_vec}
	\bdsb{\mathcal{H}}_{k,i}(\ell)
	=
	\begin{bmatrix}
		\mathbf{h}_{k,i}(\ell)\\
		\mathrm{vec}\left[\mathbf{H}_{k,i}(\ell)\right]
	\end{bmatrix},
\end{align}
with initial condition $\bdsb{\mathcal{H}}_{k,i}(0)$ given by
\begin{align}\label{h_H_definition}
	\mathbf{h}_{k,i}(0) &\triangleq \mathbf{G}_i^T(\mathbf{x}_i^k)\mathbf{g}_i(\mathbf{x}_i^k)\\
	\mathbf{H}_{k,i}(0) &\triangleq \mathbf{G}_i^T(\mathbf{x}_i^k)\mathbf{G}_i(\mathbf{x}_i^k).
\end{align}
The surrogates are the network averages of the initial conditions $\bar{\mathbf{h}}_k = \sum_{i=1}^{I}\mathbf{h}_{k,i}(0)/I$ and $\bar{\mathbf{H}}_k = \sum_{i=1}^{I}\mathbf{H}_{k,i}(0)/I$. Then, all the agents exchange their information $\bdsb{\mathcal{H}}_{k,i}(\ell)\rightarrow\bdsb{\mathcal{H}}_{k,i}(\ell+1)$ under the protocol described in Section \ref{gossip_model}.

After $\ell_k$ exchanges, the ``approximated descent" for the $(k+1)$-th update at the $i$-th agent  is
\begin{align}\label{local_descent}
	\mathbf{d}_i^k(\ell_k)
	&=  \mathbf{H}_{k,i}^{-1}(\ell_k) \mathbf{h}_{k,i}(\ell_k)
\end{align}
and the local estimate is updated as
\begin{align}\label{local_estimate}
	\mathbf{x}_i^{k+1} = P_{\mathbb{X}}\left[\mathbf{x}_i^k - \alpha\mathbf{d}_i^k(\ell_k)\right].
\end{align}

\subsection{Network Gossiping Model}\label{gossip_model}

Before describing the gossiping protocol, we first model the data exchange between different agents. We use the insights from \cite{tsitsiklis1984problems,blondel2005convergence,nedic2009distributed} and impose some rules on the agent communications over time. For each exchange, an agent $i$ communicates with its neighbor agent $j$ during $[\tau_k,\tau_{k+1})$. This is captured by a time-varying network graph $\mathcal{G}_{k,\ell}=(\mathcal{I},\mathcal{M}_{k,\ell})$ during $[\tau_{k,\ell},\tau_{k,\ell+1})$ for every GN update $k$ and gossip exchange $\ell$. The node set $\mathcal{I}=\{1,\cdots,I\}$ refers to the set of {\it agents}, and the edge set $\mathcal{M}_{k,\ell}$ is formed by the communication links in that particular gossip exchange. Associated to the graph is the adjacency matrix $\mathbf{A}_k(\ell)=[A_{ij}^{(k,\ell)}]_{I\times I}$
\begin{align}
	A_{ij}^{(k,\ell)} =
	\begin{cases}
		1, &\{i,j\}\in\mathcal{M}_{k,\ell}\\
		0, &\mathrm{otherwise}
	\end{cases}.
\end{align}

\begin{assumption}\label{connectivity_frequency}
The composite communication graph $\mathcal{G}_{k,\infty}=\{\mathcal{I},\mathcal{M}_{k,\infty}\}$ for the $k$-th update is connected, where
\begin{align*}
	\mathcal{M}_{k,\infty} \triangleq \Big\{
\{i,j\}: \{i,j\} \in \mathcal{M}_{k,\ell}~\textrm{for infinitely many}~\ell\Big\}.
\end{align*}
There exists an integer $L\geq 1$ such that\footnote{This is equivalent to the assumption that within a bounded communication interval of $L$, every agent pair $\{i,j\}$ in the composite graph communicates with each other at a frequency at least once every $L$ network exchanges.} for any $\ell$
\begin{align}
	\{i,j\} \in \bigcup_{\ell'=0}^{L-1} \mathcal{M}_{k,\ell+\ell'},\quad \forall~\{i,j\} \in \mathcal{M}_{k,\infty}.
\end{align}
\end{assumption}

With the communication model in Assumption \ref{connectivity_frequency}, each agent combines the information from its neighbors with certain weights. Define a weight matrix ${\mathbf{W}_k(\ell)}\triangleq [W_{ij}^{k}(\ell)]_{I\times I}$ for the network topology during $[\tau_{k,\ell},\tau_{k,\ell+1})$, where the $(i,j)$-th entry $W_{ij}^{k}(\ell)$ of the matrix ${\mathbf{W}_k(\ell)}$ is the weight associated to the edge $\{i,j\}$, which is non-zero if and only if $\{i,j\} \in \mathcal{M}_{k,\ell}$. 
\begin{assumption}\label{weight_matrix} For all $k$ and $\ell$, the weight matrix $\mathbf{W}_k(\ell)$ is symmetric and doubly stochastic. There exists a scalar $\eta$ with $0<\eta<1$ such that for all $i,j\in\mathcal{I}$
\begin{enumerate}
	\item $W_{ii}^{k}(\ell)\geq \eta$ for all $k>0$ and $\ell>0$.
	\item $W_{ij}^{k}(\ell)\geq \eta$ for all $k>0$ and $\ell>0$ if $\{i,j\}\in\mathcal{M}_{k,\ell}$.
	\item $W_{ij}^{k}(\ell)=0$ for all $k>0$ and $\ell>0$ if $\{i,j\}\notin\mathcal{M}_{k,\ell}$.
\end{enumerate}
\end{assumption}

The gossip exchange of each agent is local with its neighbors using this weight matrix $\mathbf{W}_k(\ell)$.
By stacking the local information vectors $\bdsb{\mathcal{H}}_k(\ell)\triangleq [\bdsb{\mathcal{H}}_{k,1}^T(\ell),\cdots,\bdsb{\mathcal{H}}_{k,I}^T(\ell)]^T$, the exchange model can be written compactly as
\begin{align}\label{gossip_exchange_matrix}
\!\!\bdsb{\mathcal{H}}_k(\ell) = \left[\mathbf{W}_k(\ell)\otimes \!\mathbf{I}_{N_{\mathcal{H}}}\right] \bdsb{\mathcal{H}}_k(\ell\!-\!1),\quad 1\leq \ell \leq \ell_k,
\end{align}
where $\ell_k$ is number of exchanges $[\tau_k,\tau_{k+1})$ as specified later.

The gossip exchange model under Assumption \ref{connectivity_frequency} and \ref{weight_matrix} is a general model that includes time-varying network formations, where all agents form random communication links with their neighbors and advance their computations of the average of all local information vectors $\bar{\bdsb{\mathcal{H}}}_k=\sum_{i=1}^{I}\bdsb{\mathcal{H}}_{k,i}(0)/I$. With the prescribed communication model, we highlight the following two special cases which are often analyzed in consensus and gossiping literature \cite{boyd2006randomized, dimakis2010gossip, johansson2008subgradient, tsitsiklis1984problems}.
\subsubsection{Coordinated Static Exchange (CSE) \cite{dimakis2010gossip,johansson2008subgradient}}\label{CSE}
In the CSE protocol, each agent combines the information from possible multiple neighbors, determined by the communication network $\mathbf{A}$, with a static weight matrix $\mathbf{W}$ for all updates and exchanges at $\tau_{k,\ell}\in[\tau_k,\tau_{k+1})$ for $\ell=1,\cdots,\ell_k$. In particular, if the network is fully connected such that $\mathbf{A}=\mathbf{I}_I-\mathbf{1}_I\mathbf{1}_I^T$, the communication interval is simply $L=1$ in which each agent talks to everybody in every exchange. There are multiple ways to choose the weight matrix in the CSE protocol, where one of the most popular choice is constructed according to the Laplacian $\mathbf{L}= \mathrm{diag}(\mathbf{A}\mathbf{1}_I)-\mathbf{A}$ as $\mathbf{W}= \mathbf{I}_I - w \mathbf{L}$ with $w=\beta/\max(\mathbf{A}\mathbf{1}_I)$ for some $0<\beta<1$.
\subsubsection{Uncoordinated Random Exchange (URE) \cite{boyd2006randomized}}\label{URE}
For each exchange in the URE protocol during $[\tau_k,\tau_{k+1})$, a random agent $i$ wakes up and chooses at random a neighbor agent $j\in\mathcal{M}_{k,
\ell}^{(i)}$ to communicate. We define the matrix $\bdsb{\Gamma}\triangleq[\gamma_{i,j}]_{I\times I}$ whose $(i,j)$-th element $\gamma_{i,j}$ represents the probability of node $i$ choosing agent $j$ once agent $i$ wakes up. The gossip exchanges are pairwise and local \cite{boyd2006randomized}. Suppose agent $I_{k,\ell}$ wakes up at $\tau_{k,\ell}\in[\tau_k,\tau_{k+1})$ and $J_{k,\ell}$ is the node picked by node $I_{k,\ell}$ with probability $\gamma_{I_{k,\ell},J_{k,\ell}}$. Then given some mixing parameter $0<\beta<1$, the weight matrix at this time is
\begin{align}
	\mathbf{W}_k(\ell) = \mathbf{I} - \beta \left(\mathbf{e}_{I_{k,\ell}} + \mathbf{e}_{J_{k,\ell}}\right)\left(\mathbf{e}_{I_{k,\ell}} + \mathbf{e}_{J_{k,\ell}}\right)^T,
\end{align}
where $\mathbf{e}_i$ is the $I$-dimensional canonical basis vector with $1$ at the $i$-th entry and $0$ otherwise. Note that the URE protocol does not necessarily satisfy Assumption \ref{connectivity_frequency},  nevertheless numerical simulations indicate that its performance degrade moderately compared to the CSE protocol.
The errors in the GGN are the topic of the following lemma:

\begin{lem}\label{weight_matrix_jroperty}
\cite[Proposition 1]{nedic2009distributed}
Let Assumption \ref{connectivity_frequency} and \ref{weight_matrix} hold. Given the minimum non-trivial weight $\eta$ in Assumption \ref{weight_matrix}, the entries of the matrix product $\prod_{\ell'=0}^{\ell}\mathbf{W}_k(\ell')$ converge with a geometric rate uniformly for all $i, j\in\mathcal{I}$ and $k$
\begin{align}
	\left|\left[\prod_{\ell'=0}^{\ell}\mathbf{W}_k(\ell')\right]_{ij} - \frac{1}{I}\right|
	\leq
	2\left(\frac{1+\eta^{-L_0}}{1-\eta^{L_0}}\right) \lambda_\eta^\ell,
\end{align}
with $L_0 =(I-1) L$ and
\begin{align}\label{lambda_eta}
    \lambda_\eta=(1-\eta^{L_0})^{1/L_0}\in(0,1).
\end{align}
\end{lem}
It is clear from Lemma \ref{weight_matrix_jroperty} that the limit of the weight matrix product exists $\underset{\ell\rightarrow\infty}{\lim}~\prod_{\ell'=0}^{\ell}\mathbf{W}_k(\ell') = \frac{1}{I}\mathbf{1}\mathbf{1}^T$ and thus we have
\begin{align}
	\underset{\ell\rightarrow\infty}{\lim}~\bdsb{\mathcal{H}}_{k,i}(\ell) = \frac{1}{I}\sum_{i=1}^{I}\bdsb{\mathcal{H}}_{k,i}(0),\quad k=1,2,\cdots
\end{align}
which asymptotically leads to the $\underset{\ell\rightarrow\infty}{\lim}~\mathbf{d}_i^k(\ell) = \bar{\mathbf{H}}_\ell^{-1}\bar{\mathbf{h}}_k$.

\begin{algorithm}[t]
\caption{Gossip-based Gauss-Newton (GGN) Algorithm}\label{GGN_algorithm}
\begin{algorithmic}[1]
\STATE {\bf given} initial variables  $\mathbf{x}_i^0$ at all agents $i\in\mathcal{I}$.
\STATE {\bf set} $k=0.$
\REPEAT
\STATE {\bf set} $k=k+1.$
\STATE {\bf initialization:} For $i\in\mathcal{I}$, each agent $i$ evaluates \eqref{h_H_definition}
			  and constructs $\bdsb{\mathcal{H}}_{k,i}(0)$ as \eqref{information_vec};
\STATE {\bf network gossiping:} Each agent $i$ exchanges information with neighbors via network gossiping as \eqref{gossip_exchange_matrix}.
\STATE {\bf local update:} For $i\in\mathcal{I}$, each agent $i$ updates its local variables as \eqref{local_estimate} and \eqref{local_descent}.
\UNTIL $\left\|\mathbf{x}_i^{k+1}-\mathbf{x}_i^k\right\|\leq \epsilon$ or $k=K$.%
\STATE {\bf set} the local estimate as $\widehat{\mathbf{x}}_i=\mathbf{x}_i^k$.
\end{algorithmic}
\end{algorithm}

\section{Convergence Analysis}\label{convergence_analysis}
In this section, we analyze the convergence of the GGN algorithm (summarized in Algorithm \ref{GGN_algorithm}) by examining the recursion in \eqref{local_estimate}. Note that the error made in the local descent \eqref{local_descent} compared with the exact descent \eqref{exact_local_descent} stems from two sources, including the gossiping error resulting from a finite $\ell_k$ and the mismatch error by using the surrogates $\bar{\mathbf{h}}_k$ and $\bar{\mathbf{H}}_k$ instead of the exact quantities. In the following, we analyze the effect of this error in the convergence of the GGN algorithm.

\subsection{Perturbed Recursion Analysis}
At the $(k+1)$-th update, the error between the local estimate $\mathbf{x}_i^{k+1}$ and a fixed point in \eqref{fixed_joint} satisfies the following recursion.
\begin{lem}\label{lem_error_recursion}
Let $\mathbb{X}$ be a compact convex set and Assumption \ref{lipshitz} hold. The error $\|\mathbf{x}_i^{k+1}-\mathbf{x}^\star\|$ between the local iterate $\mathbf{x}_i^k$ at each update \eqref{local_estimate} and an arbitrary fixed point $\mathbf{x}^\star$ in \eqref{fixed_joint} satisfies the following recursion
\begin{align}\label{distributed_error_recursion}
	&\left\|\mathbf{x}_i^{k+1}-\mathbf{x}^\star\right\|\\
	&\leq
	T_1\left\|\mathbf{x}_i^k-\mathbf{x}^\star\right\|^2
	+T_2\left\|\mathbf{x}_i^k-\mathbf{x}^\star\right\|
    +\alpha\|\mathbf{d}_i^k(\ell_k)-\mathbf{d}_i^k\|,\nonumber
\end{align}
where
\begin{align}\label{T_constants}
    T_1 &\triangleq \frac{\alpha\omega}{2\sigma_{\min}},\quad
     \epsilon_{\min}  \triangleq \left\|\mathbf{g}(\mathbf{x}^\star)\right\|\\
    T_2 &\triangleq (1-\alpha)\frac{\sigma_{\max}}{\sigma_{\min}}+\frac{\sqrt{2}\alpha\omega\epsilon_{\min}}{\sigma_{\min}^2}.
\end{align}
\end{lem}
\begin{proof}
	See Appendix \ref{proof_lem_error_recursion}.
\end{proof}
The error recursion is a perturbed version of the centralized recursion. The discrepancy  between the distributed and centralized update is $\|\mathbf{d}_i^k(\ell_k)-\mathbf{d}_i^k\|$, and its convergence is analyzed in the following theorem.

\begin{thm}\label{theorem_convergence}	
({Convergence with Bounded Perturbation}) Let Assumption \ref{lipshitz} hold and $\mathbb{X}$ be a compact convex set. Given a step-size chosen as
\begin{align}
    \max\left\{1-\frac{3\sigma_{\min}}{\sigma_{\max}},0\right\}<\alpha\leq 1
\end{align}
and the condition
\begin{align}
    \omega\epsilon_{\min} < \frac{\sigma_{\min}^2}{\sqrt{2}\alpha}\left[ 3-(1-\alpha)\frac{\sigma_{\max}}{\sigma_{\min}}\right],
\end{align}
we define the following
\begin{align}\label{rho_def}
	\rho_{\min} &= \frac{(1-T_2) - \sqrt{(1-T_2)^2-4\alpha T_1\kappa}}{2T_1}\\
	\rho_{\max} &= \frac{(1-T_2) + \sqrt{(1-T_2)^2-4\alpha T_1\kappa}}{2T_1}
\end{align}
where $\kappa$ is a bounded perturbation with
\begin{align}
    0<\kappa < \frac{(1-T_2)^2}{4\alpha T_1}.
\end{align}
If the  $\|\mathbf{d}_i^k(\ell_k)-\mathbf{d}_i^k\| \leq \kappa$ is bounded for all $k$ and $i\in\mathcal{I}$, then given any $\mathbf{x}_i^0$ that falls within the $\rho_{\max}$-neighborhood of a certain fixed point $\mathbf{x}^\star\in\mathbb{X}$
\begin{align}\label{initial_condition_requirement}
	\left\|\mathbf{x}_i^0-\mathbf{x}^\star\right\|<\rho_{\max},
\end{align}
the asymptotic error of the local iterate $\mathbf{x}_i^k$ at each agent with respect to $\mathbf{x}^\star$ can be bounded as
\begin{align}
	\underset{k\rightarrow\infty}{\limsup} \left\|\mathbf{x}_i^{k+1}-\mathbf{x}^\star\right\|
	\leq \rho_{\min}.
\end{align}
\end{thm}
\begin{proof}
	See Appendix \ref{proof_theorem_convergence}.
\end{proof}

An intuition that can be drawn from the sufficient condition is that the smaller is the Lipschitz constant $\omega$, the larger is the region of convergence around the fixed point $\mathbf{x}^\star$ one can start with. In other words, the smoother the cost function the better the convergence. If $\epsilon_{\min}$ in \eqref{T_constants} is small (e.g., the fixed point is the minimizer $\widehat{\mathbf{x}}$ in \eqref{minimizer}), then by letting $\alpha=1$ and assuming $\kappa\rightarrow 0$, we have $\rho_{\max}\approx 2\sigma_{\min}/\omega-\kappa$ and the steady error is approximately $\rho_{\min}\approx \kappa$ with finite iterations, which scales with the gossiping error. Furthermore, when $\epsilon_{\min}=0$ the convergence rate is quadratic, same as the Newton's method. Finally, when $\kappa=0$, the result reduces to the centralized Gauss-Newton algorithm since there is no perturbation.

\subsection{Perturbation Analysis of $\kappa$}
Given that the perturbation is bounded, Theorem \ref{theorem_convergence} is sufficient to guarantee convergence of the GGN algorithm. In the following, we analyze this perturbation and show that the bounded condition holds.

\subsubsection{Gossiping error}
Define at the $\ell$-th exchange
\begin{align*}
	\mathbf{h}_k(\ell) &\triangleq [\mathbf{h}_{k,1}^T(\ell), \cdots, \mathbf{h}_{k,I}^T(\ell)]^T,\\
	\mathbf{H}_k(\ell) &\triangleq [\mathbf{H}_{k,1}^T(\ell),\cdots,\mathbf{H}_{k,I}^T(\ell)]^T
\end{align*}	
and their deviations from the exact averages $ \bar{\mathbf{h}}_k$ and $\bar{\mathbf{H}}_k$ as
\begin{align*}
	\mathbf{e}_k(\ell) &= [\mathbf{e}_{k,1}^T(\ell),\cdots,\mathbf{e}_{k,I}^T(\ell)]^T,\\
	\mathbf{E}_k(\ell) &= [\mathbf{E}_{k,1}^T(\ell),\cdots,\mathbf{E}_{k,I}^T(\ell)]^T,
\end{align*}
where $\mathbf{e}_{k,i}(\ell)=\mathbf{h}_{k,i}(\ell)  - \bar{\mathbf{h}}_k$ and $\mathbf{E}_{k,i}(\ell)=\mathbf{H}_{k,i}(\ell)  - \bar{\mathbf{H}}_k$.
The gossip errors $\mathbf{e}_k(\ell_k)$ and $\mathbf{E}_k(\ell_k)$ are related to the properties of the weight matrices $\mathbf{W}_k(\ell)$ in Lemma \ref{weight_matrix_jroperty}.
\begin{lem}\label{lem_gossip_error}
Let Assumption \ref{connectivity_frequency} and \ref{weight_matrix} hold. The gossip error $\mathbf{e}_k(\ell_k)$ and $\mathbf{E}_k(\ell_k)$ after the $k$-th update can be bounded as
\begin{align*}
	\left\|\mathbf{e}_k(\ell_k)\right\| &< C \lambda_\eta^{\ell_k}, \quad
	\left\|\mathbf{E}_k(\ell_k)\right\|_F < C \lambda_\eta^{\ell_k},
\end{align*}
where
\begin{align}\label{C}
    C \triangleq 2I\sigma_{\max}\sqrt{I( \epsilon_{\max} ^2 + N\sigma_{\max}^2)} \left(\frac{1+\eta^{-L_0}}{1-\eta^{L_0}}\right).
\end{align}
\end{lem}
\begin{proof}
	See Appendix \ref{proof_lem_gossip_error}.
\end{proof}

\subsubsection{Mismatch of surrogates}

Define the errors between the surrogate $\bar{\mathbf{h}}_k$, $\bar{\mathbf{H}}_k$ and exact quantities $\mathbf{q}(\mathbf{x}_i^k)$ and $\mathbf{Q}(\mathbf{x}_i^k)$ as

\begin{align}\label{definition_delta_Delta}
	\bdsb{\delta}_{k,i} &= \bar{\mathbf{h}}_k - \mathbf{q}(\mathbf{x}_i^k) = \frac{1}{I}\sum_{j=1}^{I}\left[\mathbf{h}_{k,i}(\ell)-\mathbf{h}_{k,j}(\ell)\right]\\
	\bdsb{\Delta}_{k,i} &= \bar{\mathbf{H}}_k - \mathbf{Q}(\mathbf{x}_i^k) = \frac{1}{I}\sum_{j=1}^{I}\left[\mathbf{H}_{k,i}(\ell)-\mathbf{H}_{k,j}(\ell)\right],\nonumber
\end{align}
which thus leads to
\begin{align}\label{h_H_decomp}
	\mathbf{h}_{k,i}(\ell) &= \mathbf{q}(\mathbf{x}_i^k) + \bdsb{\delta}_{k,i} +\mathbf{e}_{k,i}(\ell),\\
	\mathbf{H}_{k,i}(\ell) &= \mathbf{Q}(\mathbf{x}_i^k) + \bdsb{\Delta}_{k,i} + \mathbf{E}_{k,i}(\ell).
\end{align}

By \eqref{definition_delta_Delta} and Corollary \ref{ass_lipschitz}, we have
\begin{align}\label{lipschitz_delta_Delta}
	\left\|\bdsb{\delta}_{k,i}\right\|
	& \leq \frac{\nu_{\delta}}{I}\sum_{j=1}^I \left\|\mathbf{x}_i^k-\mathbf{x}_j^k\right\|\\
	\left\|\bdsb{\Delta}_{k,i}\right\|
	& \leq \frac{\nu_{\Delta}}{I}\sum_{j=1}^I \left\|\mathbf{x}_i^k-\mathbf{x}_j^k\right\|.
\end{align}


Clearly, this discrepancy depends on the disagreement $\left\|\mathbf{x}_i^k-\mathbf{x}_j^k\right\|$ for each pair of $i$-th and $j$-th agents, characterized by the mismatch $\bdsb{\Delta}_{k,i}$ and $\bdsb{\delta}_{k,i}$ which originates from the gossip errors $\mathbf{E}_{k,i}(\ell_k)$ and $\mathbf{e}_{k,i}(\ell_k)$. Having analyzed the gossip error dynamics in Lemma \ref{lem_gossip_error}, in the following we bound the disagreement $\left\|\mathbf{x}_i^k-\mathbf{x}_j^k\right\|$.

\begin{assumption}\label{ell_ass}
Denote by $\ell_{\min} = \min_k\{\ell_k\}$ the minimum exchange. We assume that $\{\ell_k\}_{k=0}^{\infty}$ are chosen to satisfy\footnote{A simple choice is $\ell_0=\ell_{\min}$ and $\ell_k=\ell_{k-1}+1$, then $\lambda_\infty=1/(1-\lambda_\eta)$.}
\begin{align*}
    \lambda_\infty\triangleq\underset{K\rightarrow \infty}{\lim}~\sum_{k=0}^{K} \lambda_\eta^{(\ell_k-\ell_{\min})}< \infty.
\end{align*}
For any $\xi\in(0,1/2)$, the number $\ell_{\min}$ is chosen as
\begin{align}\label{denominator}
	\ell_{\min} &= \left\lceil\log\left(\frac{\xi}{4D}\right)/\log\lambda_\eta\right\rceil\\
    D &\triangleq CC_2 (\nu\lambda_\infty C_1C_2 + 1)
\end{align}
where $C,\lambda_\eta$ are defined in \eqref{C} \eqref{lambda_eta}, $\nu=\max\{\nu_\delta,\nu_\Delta\}$ and
\begin{align}\label{C_infty}
    C_1 \triangleq 2\left(1+\frac{\sigma_{\max} \epsilon_{\max} }{\sigma_{\min}^2}\right),\quad
    C_2=\frac{I}{\sigma_{\min}^2}
\end{align}
with $ \epsilon_{\max} $, $\sigma_{\min}$ and $\sigma_{\max}$ given by Assumption \ref{lipshitz}.
\end{assumption}

\begin{lem}\label{lem_disagreement}
Let the minimum exchange $\ell_{\min}$ be chosen based on an arbitrary value $\xi\in(0,1/2)$ using \eqref{denominator}. According to Lemma \ref{weight_matrix_jroperty} under Assumption \ref{lipshitz}, \ref{connectivity_frequency}, \ref{weight_matrix} and \ref{ell_ass}, then if the initializer is the same for all agents $\mathbf{x}_i^0=\mathbf{x}^0$, the deviation $\left\|\mathbf{x}_i^K-\mathbf{x}_j^K\right\|$ for any $i$ and $j$ at the $K$-th update satisfies
\begin{align}
	\left\|\mathbf{x}_i^K-\mathbf{x}_j^K\right\|
	&\leq \xi \left(\frac{CC_1C_2}{D}\right) \sum_{k=0}^{K-1}\lambda_\eta^{(\ell_k-\ell_{\min})},
\end{align}
where $C$ is the gossip error scale in \eqref{C}, $C_1,C_2$ are defined in \eqref{C_infty} and $\lambda_\eta$ is the gossip convergence rate in \eqref{lambda_eta}. Based on Assumption \ref{ell_ass}, this implies
\begin{align*}
	\left\|\mathbf{x}_i^K-\mathbf{x}_j^K\right\|
	&\leq 4 CC_1C_2 \sum_{k=0}^{K-1}\lambda_\eta^{\ell_k+1}.
\end{align*}
\end{lem}
\begin{proof}
	See Appendix \ref{proof_lem_disagreement}.
\end{proof}

\begin{thm}\label{proposition_discrepancy}
Under Assumption  \ref{lipshitz}, \ref{connectivity_frequency}, \ref{weight_matrix} and \ref{ell_ass}, Given Lemma \ref{weight_matrix_jroperty}, \ref{lem_gossip_error} and \ref{lem_disagreement}, the discrepancy between the inexact and the exact descent is bounded for all $i$ and $k$
\begin{align}
	\left\|\mathbf{d}_i^k(\ell_k) - \mathbf{d}_i^k\right\|
	<\kappa,
\end{align}
by the finite perturbation $\kappa \triangleq 4 C_1 D \lambda_\eta^{(\ell_{\min}+1)}$, whose magnitude vanishes exponentially with respect to the minimum number of gossip exchanges $\lim_{\ell_{\min}\rightarrow\infty} \kappa = 0$.
\end{thm}
\begin{proof}
	See Appendix \ref{proof_jroposition_discrepancy}.
\end{proof}

Theorem \ref{theorem_convergence} and \ref{proposition_discrepancy} indicate that if the exchanges $\ell_k$'s are large, then $\kappa\rightarrow 0$ and the recursion approaches the centralized version.  Note that Theorem \ref{theorem_convergence} and \ref{proposition_discrepancy} are proven using very pessimistic bounds. In Section \ref{numerical_results} the numerical results show the algorithm behaves well even with link failures, in spite of not meeting all the conditions and assumption stated.


\section{Application : Power System State Estimation}\label{PSSE}
A power network is characterized by vertices (called {\it buses}) representing simple interconnections, generators or loads, denoted by the set $\mathcal{N}\triangleq\{1,\cdots,N\}$. Transmission lines connecting these buses constitute the power grid topology, denoted by the edge set $\mathcal{E}$ with cardinality $|\mathcal{E}|=L$. 
The electrical parameters of the grid are characterized by the admittance matrix $\mathbf{Y}=[-Y_{nm}]_{N\times N}$, where $Y_{nm}=G_{nm}+\mathrm{i}B_{nm}$, $\{n,m\}\in\mathcal{E}$ is the line admittance, and the shunt admittance $\bar{Y}_{nm}=\bar{G}_{nm}+\mathrm{i}\bar{B}_{nm}$ associated with the $\Pi$-model\footnote{The $\Pi$-model is a circuit equivalent of a transmission line by abstracting two electric buses as a two-port network in the shape of a $\Pi$ connection \cite{monticelli1999state}.} of each transmission line $\{n,m\}\in\mathcal{E}$. Note that $Y_{nn}= - \sum_{l\neq n} (\bar{Y}_{nm} + Y_{nm})$ is defined as the self-admittance. 
The state of the power system corresponds to the voltage phasors at all buses, described by voltage phase and magnitude $\mathbf{x}=[\bdsb{\Theta}^T,\mathbf{V}^T]^T$, where $\bdsb{\Theta}\triangleq[\theta_1, \cdots,\theta_N]^T$ is the phase vector with $\theta_1$ being the slack bus phase, and $\mathbf{V}\triangleq \left[V_1,\cdots,V_N\right]^T$ contains the magnitude.

\begin{figure}
\centering
\includegraphics[width=0.9\linewidth]{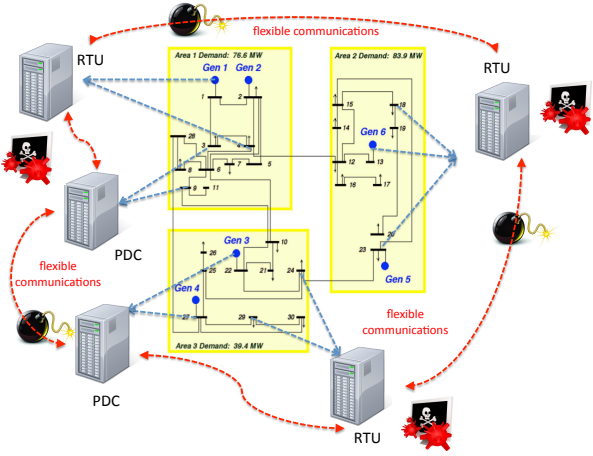}
\caption{Multi-site structure in IEEE-30 test case}\label{fig.IEEE_30_division}
\end{figure}

\subsection{Power Measurement Models}
Power measurements include the {\it active/reactive power injection} $(P_n,Q_n)$ for buses $n\in\mathcal{N}$, and the {\it active/reactive power flows} $(P_{nm},Q_{nm})$ on transmission lines $(n,m)\in\mathcal{E}$. The ensemble of these quantities can be stacked into the length-$2N$ power injection vector $\mathbf{f}_{\mathcal{I}}(\mathbf{x})$, as well as the length-$4L$ line flow vector $\mathbf{f}_{\mathcal{F}}(\mathbf{x})$ respectively
\begin{align}
	\mathbf{f}_{\mathcal{I}}(\mathbf{x}) &= [P_1(\mathbf{x}),\cdots,P_N(\mathbf{x}), Q_1(\mathbf{x}),\cdots,Q_N(\mathbf{x})]^T\\
	\mathbf{f}_{\mathcal{F}}(\mathbf{x}) &= [\cdots,P_{nm}(\mathbf{x}),\cdots,\cdots,Q_{nm}(\mathbf{x}), \cdots]^T
\end{align}
and expressed in relation to the state $\mathbf{x}$ as in \cite{monticelli1999state} 
\begin{align*}
	P_n(\mathbf{x})
	&= V_n\sum_{m\neq n}^{N} V_m \left( G_{nm}\cos\theta_{nm} + B_{nm} \sin\theta_{nm}\right)\\
	Q_n(\mathbf{x})
	&= V_n\sum_{m\neq n}^{N} V_m \left( G_{nm}\sin\theta_{nm} - B_{nm} \cos\theta_{nm}\right)\\
	P_{nm}(\mathbf{x})
	&= V_n^2G_{nm} - V_nV_m \left(G_{nm}\cos\theta_{nm} + B_{nm} \sin\theta_{nm}\right)\\
	Q_{nm}(\mathbf{x})
	&= -V_n^2B_{nm} - V_nV_m \left(G_{nm}\sin\theta_{nm} - B_{nm} \cos\theta_{nm}\right),
\end{align*}
where $\theta_{nm}=\theta_n-\theta_m$. 
By stacking the power flow equations and the measurements into vectors $\mathbf{f}(\mathbf{x}) \triangleq [ \mathbf{f}_{\mathcal{I}}^T(\mathbf{x}),\mathbf{f}_{\mathcal{F}}^T(\mathbf{x})]^T$ and $\mathbf{z}  \triangleq [\mathbf{z}_{\mathcal{I}}^T, \mathbf{z}_{\mathcal{F}}^T]^T$, the measurement ensemble in the presence of measurement error $\bdsb{\varepsilon}\triangleq [ \bdsb{\varepsilon}_{\mathcal{I}}^T, \bdsb{\varepsilon}_{\mathcal{F}}^T]^T$ is
\begin{align}\label{meas-model_all}
	\mathbf{z} = \mathbf{f}(\bar{\mathbf{x}}) + \bdsb{\varepsilon},
\end{align}
where $\bar{\mathbf{x}}=[\bar{\theta}_1,\cdots,\bar{\theta}_N,\bar{V}_1,\cdots,\bar{V}_N]^T$ is the true state.

%
%

\subsection{Formulation and Solution for the PSSE}
A reasonable abstraction of the data acquisition architecture in power systems is as an interconnected multi-site infrastructure, with $I$ {\it sites} in which the $i$-th {\it site} covers a subset of buses $n\in\mathcal{N}_i$ satisfying $\mathcal{N}_j\bigcap \mathcal{N}_i = \varnothing$ and $\mathcal{N}_i,\mathcal{N}_j\subset\mathcal{N}$ (Fig. \ref{fig.IEEE_30_division}). The $i$-th {\it site} temporally aligns and aggregates a snapshot of $M_i$ local measurements of $\{z_{i,m}\}_{m=1}^{M_i}$ within the site or on the lines that connect its own site with others. The $i$-th site's measurements are selected from the ensemble in \eqref{meas-model_all} as
\begin{align}\label{meas-model}
	\mathbf{z}_i = \mathbf{T}_i\mathbf{z} = \mathbf{f}_i(\mathbf{x}) + \mathbf{T}_i\bdsb{\varepsilon},
\end{align}
where $\mathbf{f}_i(\mathbf{x})\triangleq \mathbf{T}_i\mathbf{f}(\mathbf{x})$, and $\mathbf{T}_i \triangleq \mathrm{diag}[
\mathbf{T}_{i,\mathcal{I}},\mathbf{T}_{i,\mathcal{F}}]$ is a block diagonal binary matrix selecting the corresponding measurements at the $i$-th site. Specifically, $\mathbf{T}_{i,\mathcal{I}}\in\{0,1\}^{M_{i,\mathcal{I}}\times 2N}$ and $\mathbf{T}_{i,\mathcal{F}}\in\{0,1\}^{M_{i,\mathcal{F}}\times 4L}$ are selection matrices with each row having only one ``$1$" entry located at the index of the corresponding element in $\mathbf{f}(\mathbf{x})$ measured by field devices. The number of measurements recorded by each agent is $M_i=
M_{i,\mathcal{I}}+M_{i,\mathcal{F}} = \mathrm{Tr}(\mathbf{T}_i^T\mathbf{T}_i)$.

The universally accepted problem formulation for static state estimation is to solve a weighted NLLS problem that fits the estimated state to the power measurements \cite{schweppe1974static, larson1970state}. Assuming $\mathbb{E}\{\bdsb{\varepsilon}\bdsb{\varepsilon}^T\}=\sigma^2\mathbf{I}$, the state is estimated as 
\begin{align}\label{centralized_SE}
	 \widehat{\mathbf{x}} = \underset{\mathbf{x}\in\mathbb{X}}{\min}&
	 ~~\sum_{i=1}^{I} \left\|\mathbf{z}_i-\mathbf{f}_i(\mathbf{x})\right\|^2
\end{align}
where $\mathbb{X} \triangleq \left\{ \theta_n \in [-\theta_{\max},\theta_{\max}], V_n\in [0, V_{\max}], n\in\mathcal{N}\right\}$ with $\theta_{\max}$ and $V_{\max}$ being the phase angle and voltage limit. By letting $\mathbf{g}_i(\mathbf{x}) \triangleq \mathbf{z}_i-\mathbf{f}_i(\mathbf{x})$ and $\mathbf{G}_i(\mathbf{x}) \triangleq -{\partial\mathbf{f}_i(\mathbf{x})}/{\partial\mathbf{x}^T}$, the problem can be solved using the proposed GGN algorithm.

In many practical scenarios \cite{kar2008distributed,cattivelli2008diffusion,cattivelli2010diffusion,lopes2008diffusion}, many similar NLLS problems in a network take the form
\begin{align}\label{adaptive}
	\widehat{\mathbf{x}}[t] = \underset{\mathbf{x}\in\mathbb{X}}{\min}&
	 ~~\sum_{i=1}^{I} \left\|\mathbf{z}_i[t]-\mathbf{f}_i(\mathbf{x})\right\|^2,
\end{align}
where $\mathbf{z}_i[t]\in\mathbb{R}^{M_i}$ is a snapshot of measurements taken at agent $i$ at time $t$. In this scenario, the GGN algorithm can be readily applied to track the state by initializing $\mathbf{x}_i^0[t]$ with the previous local estimate $\widehat{\mathbf{x}}_i[t]$. In the following, we show the performance of the GGN algorithm in estimating and tracking the state of power systems using real-time measurements.

\section{Numerical Results}\label{numerical_results}

\nocite{UK_grid}

In this section, we compare the GGN algorithm cost in \eqref{centralized_SE} and Mean Square Error (MSE) with that of the algorithms \cite{xie2012fully}. We also show numerically that the GGN algorithm can process measurement adaptively and compare it against the method in \cite{kar2008distributed}. Given the distributed estimates $\{\widehat{V}_{i,n}^{(k)},\widehat{\theta}_{i,n}^{(k)}\}$ at each update, the local MSE relative to $\bar{V}_n$'s and $\bar{\theta}_n$'s is
\begin{align}
    \mathrm{MSE}_{V,i}^{(k)}  &= \mathbb{E}\left\{ \sum_{n=1}^N (\widehat{V}_{i,n}^{(k)}-\bar{V}_n)^2 \right\}/N\label{MSE_V}\\
    \mathrm{MSE}_{\Theta,i}^{(k)}  &=  \mathbb{E}\left\{ \sum_{n=1}^N (\widehat{\theta}_{i,n}^{(k)}-\bar{\theta}_n)^2  \right\}/N.\label{MSE_Theta}
\end{align}
The results are averaged over $10^3$ experiments. We also show the MSE  of the GGN algorithm using the URE protocol in the presence of random link failures.

We considered the IEEE-30 bus ($N=30$) model in MATPOWER 4.0. The initialization is $1$ for voltage magnitude and $0$ for the phase. We take one snapshot of the load profile from the UK National Grid load curve from \cite{UK_grid} and scale the base load from MATPOWER on the load buses. Then we run the Optimal Power Flow (OPF) program to determine the generation dispatch for that snapshot. This gives us the true state $\bar{\mathbf{x}}$ and $\mathbf{f}(\bar{\mathbf{x}})$ in per unit (p.u.) values. Measurements $\{\mathbf{z}_i\}_{i=1}^{I}$ by are generated adding independent Gaussian errors $\varepsilon_{i,m}\sim \mathcal{N}(0,\sigma^2)$ with $\sigma^2=10^{-6}$.

\begin{figure}[t]
\begin{center}
{\subfigure[][Objective value]{\resizebox{0.45\textwidth}{!}{\includegraphics{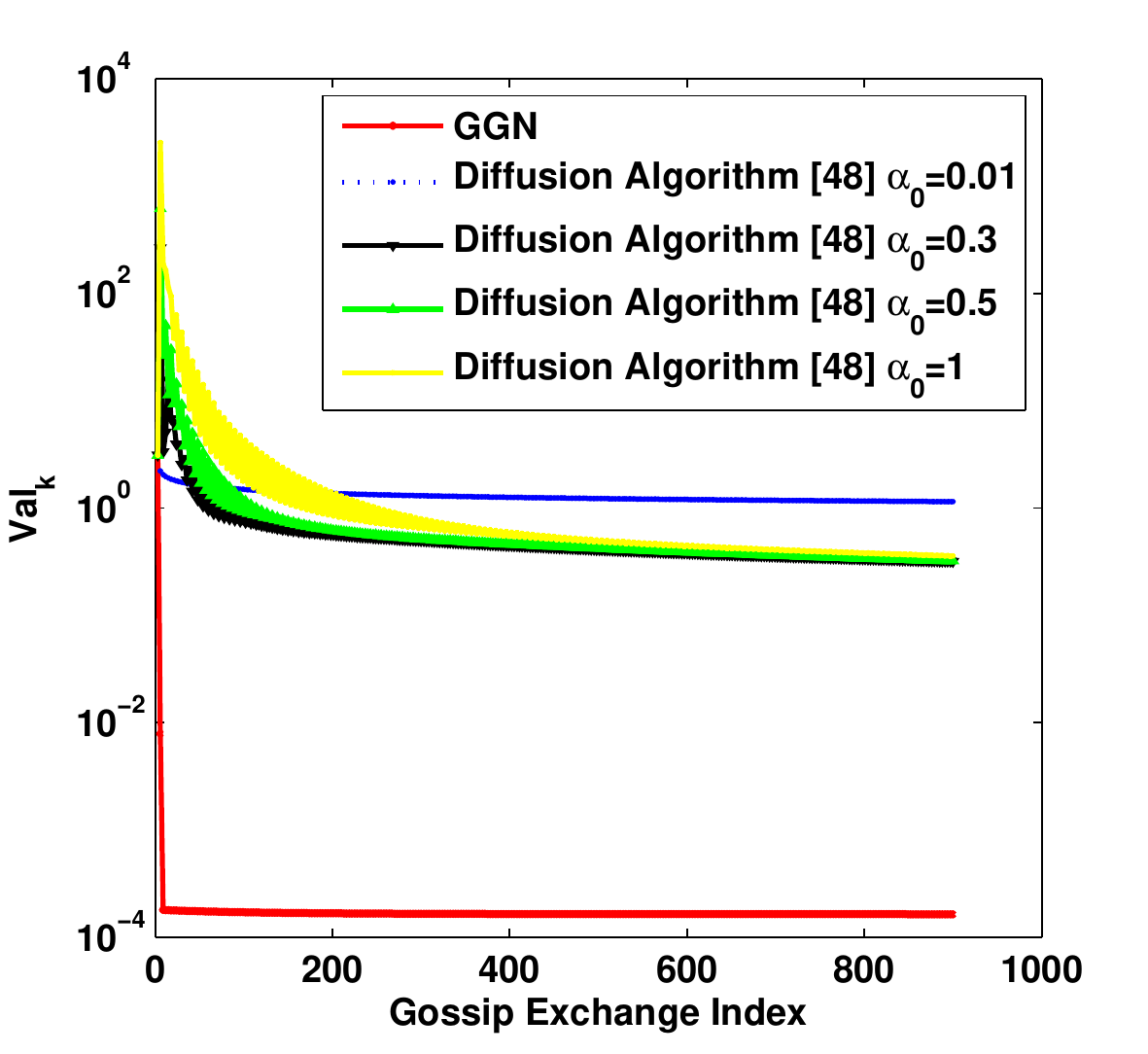}\label{fig.diff_obj1}}}}\\
\vspace{-0.5cm}
{\subfigure[][Gradient norm]{\resizebox{0.45\textwidth}{!}{\includegraphics{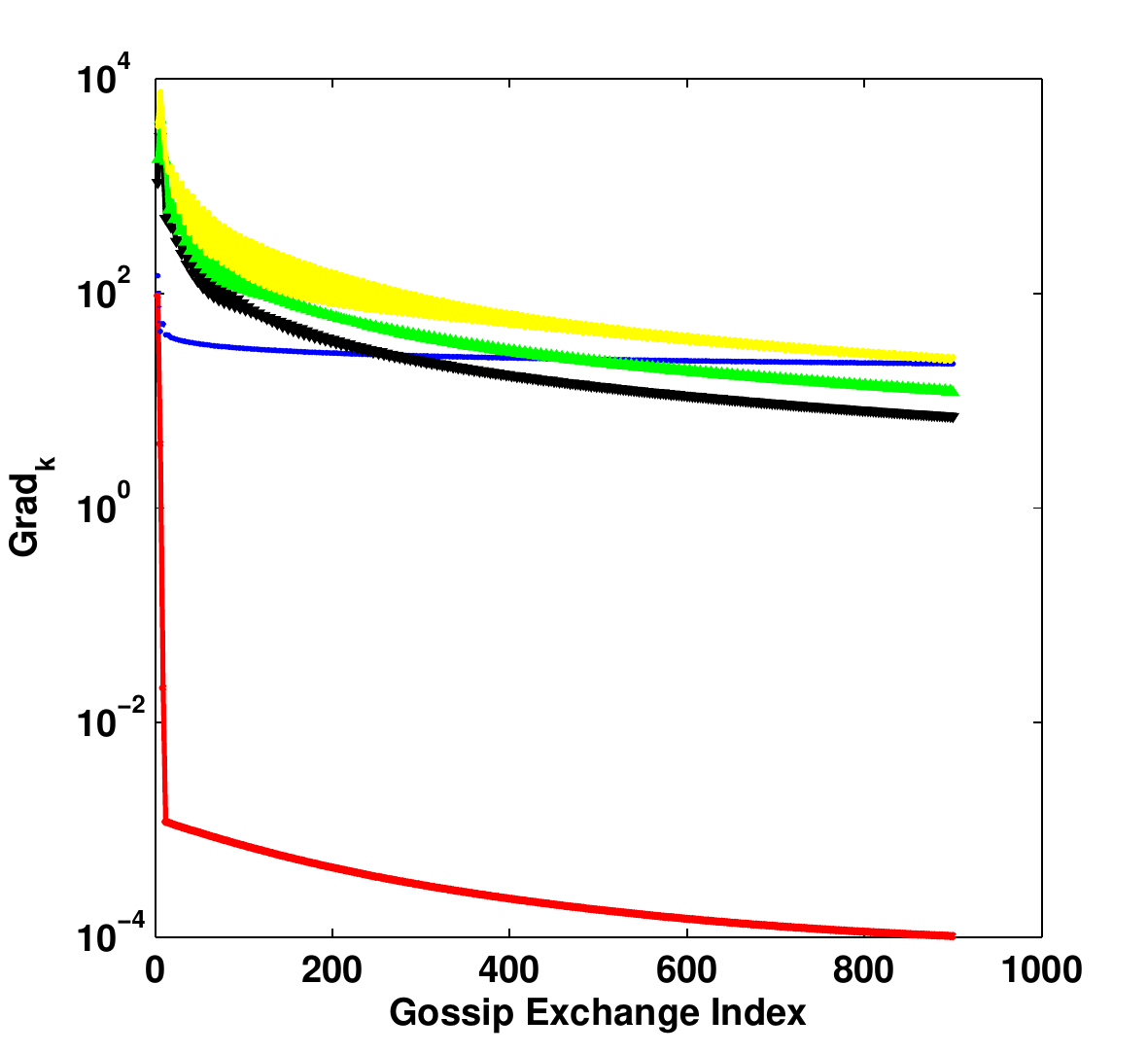}\label{fig.diff_grad}}}}
\end{center}\vspace{-0.1cm}
\caption{Comparison between the GGN algorithm (CSE Protocol) and diffusion algorithm in \cite{xie2012fully} with $\ell_{\min}=3$ exchanges}
\vspace{-0.4cm}
\end{figure}

\subsection{Comparison with Diffusion Algorithms under CSE Protocol}
Here we evaluate the performance of the GGN algorithm against the diffusion algorithm for PSSE in \cite{xie2012fully}, and its extension to adaptive processing in \cite{kar2008distributed}. To make a fair comparison in terms of communication costs and accuracy, we exploit the CSE protocol used in \cite{xie2012fully,kar2008distributed} for our method, where the exchange is coordinated and synchronous. For simplicity, we divide the system into $I=3$ sites as in Fig. \ref{fig.IEEE_30_division} and the communication graph is fully connected, giving an adjacency matrix $\mathbf{A}= \mathbf{1}_I\mathbf{1}_I^T-\mathbf{I}$. The weight matrix is constructed according to the Laplacian $\mathbf{L}= \mathrm{diag}(\mathbf{A}\mathbf{1}_I)-\mathbf{A}$ as $\mathbf{W}= \mathbf{I}_I - w \mathbf{L}$ with $w=\beta/\max(\mathbf{A}\mathbf{1}_I)$ and $\beta = 0.3$. The step-size is $\alpha_{\rm GGN}=0.5$ for the GGN algorithm while $\alpha_{\textrm{diff},\ell} = 0.01 \ell^{-1}, 0.3 \ell^{-1}, 0.5 \ell^{-1}, \ell^{-1}$ for \cite{xie2012fully,kar2008distributed}. The network diffusion algorithm takes place at each exchange $\ell$, while the GGN algorithm runs $\ell_k=\ell_0=\ell_{\min}=3$ exchanges for each update. Therefore, the exchange index $\ell$ in the GGN algorithm is the remainder of the index $\ell$ in the diffusion algorithm divided by $3$.

\subsubsection{Estimation on Static Measurements}
In this subsection, we show the comparison between our approach and that in \cite{xie2012fully} over $900$ exchanges overall. In particular, the comparison is on the global objective \eqref{centralized_SE} evaluated with local estimates
\begin{align}\label{Val}
	\mathrm{Val}_k = \sum_{i=1}^{I} \left\|\mathbf{z}_i-\mathbf{f}_i(\mathbf{x}_i^k)\right\|^2
\end{align}
and the following term to evaluate the optimality in \eqref{fixed_joint}
\begin{align}
	\mathrm{Grad}_k=\sum_{i=1}^I \left\|\mathbf{G}_i^T(\mathbf{x}_i^k)(\mathbf{z}_i-\mathbf{f}_i(\mathbf{x}_i^k))\right\|.
\end{align}
which are plotted against the total number of gossip exchanges so that the comparison is performed on the same time scale.

\begin{figure*}[!t]
\begin{center}
{\subfigure[][Objective value]{\resizebox{0.32\textwidth}{!}{\includegraphics{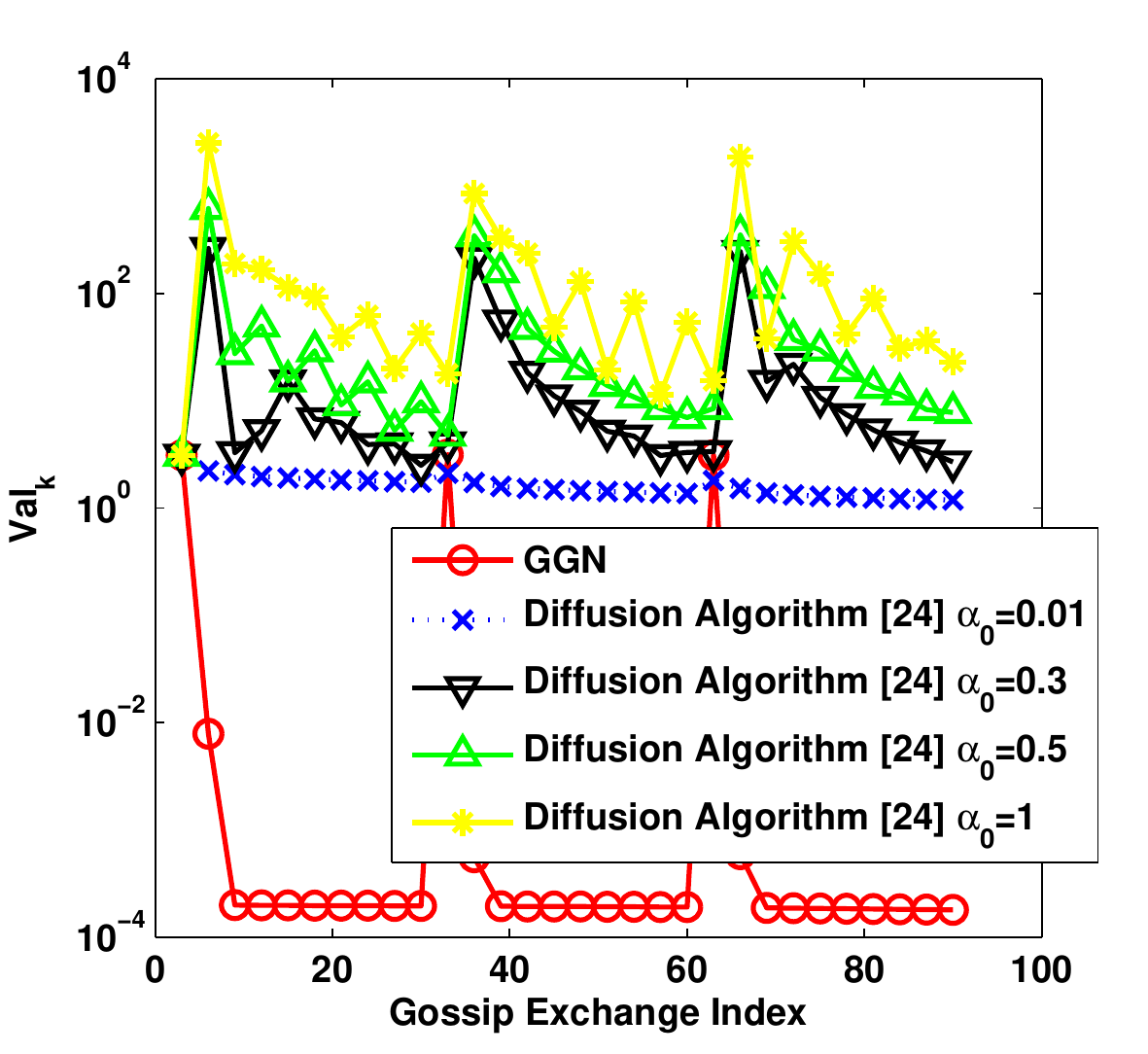}\label{fig.diff_obj_dynamic}}}}
\vspace{-0.2cm}
{\subfigure[][$\mathrm{MSE}_V$ Comparison]{\resizebox{0.32\textwidth}{!}{\includegraphics{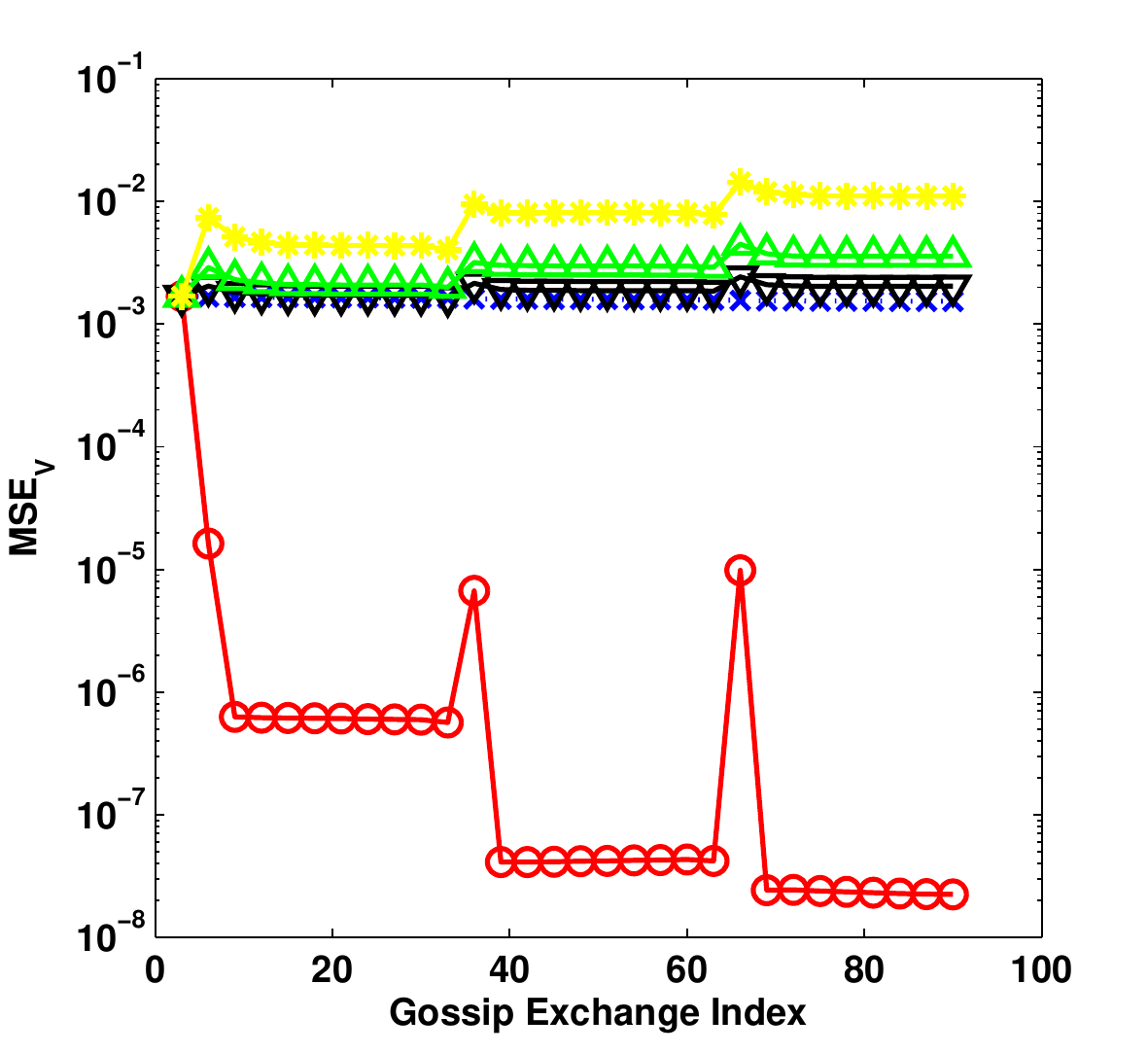}\label{fig.diff_Vm_dynamic}}}}
\vspace{-0.2cm}
{\subfigure[][$\mathrm{MSE}_\Theta$ Comparison]{\resizebox{0.32\textwidth}{!}{\includegraphics{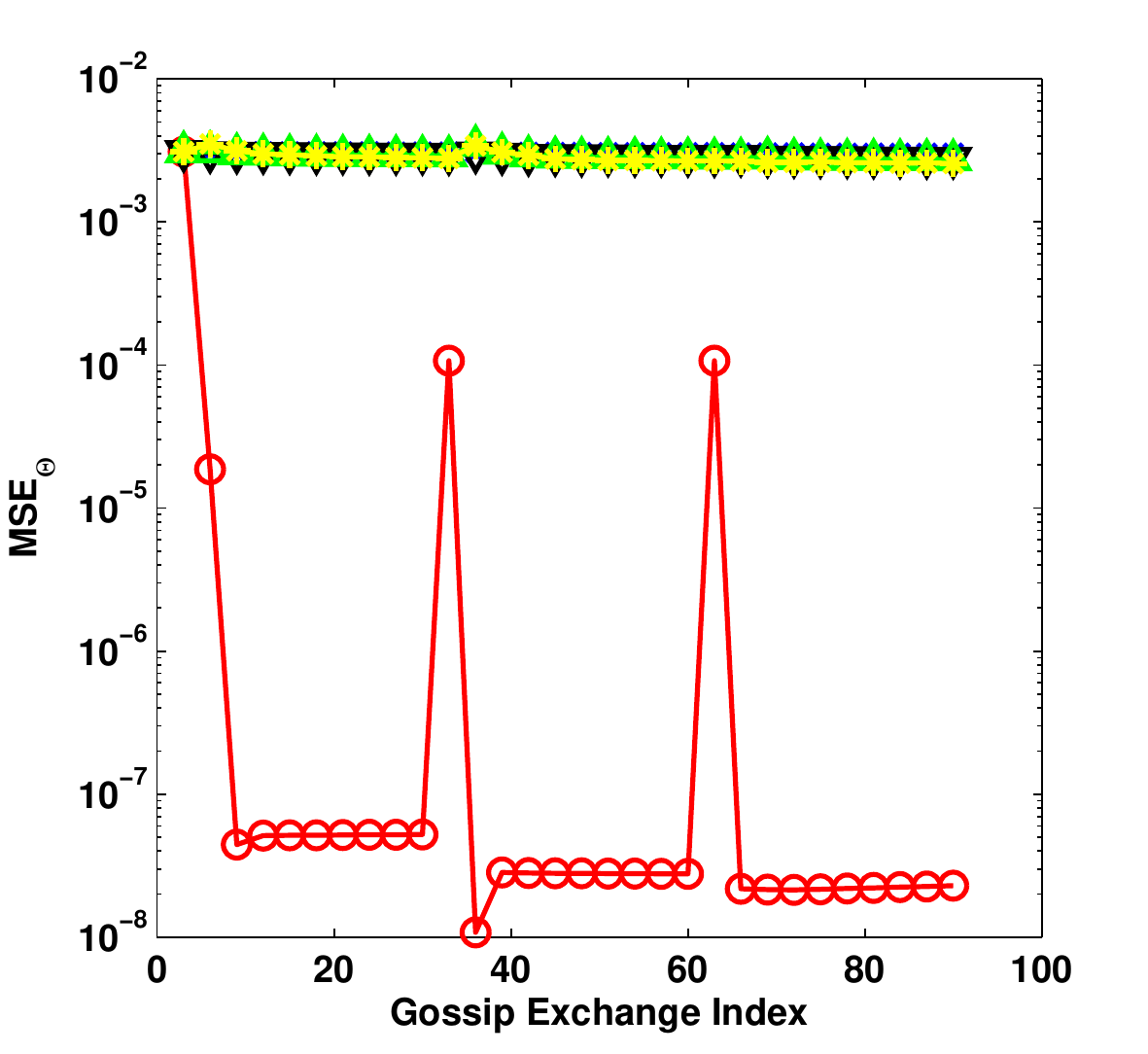}\label{fig.diff_Va_dynamic}}}}
\vspace{0.2cm}
\end{center}
\vspace{-0.2cm}
\caption{Comparison between GGN Algorithm (CSE Protocol) and \cite{kar2008distributed} against $\ell$ with $\ell_{\min}=3$ exchanges for every update.}
\end{figure*}

\begin{figure*}[!t]
\begin{center}
{\subfigure[][Objective value]{\resizebox{0.32\textwidth}{!}{\includegraphics{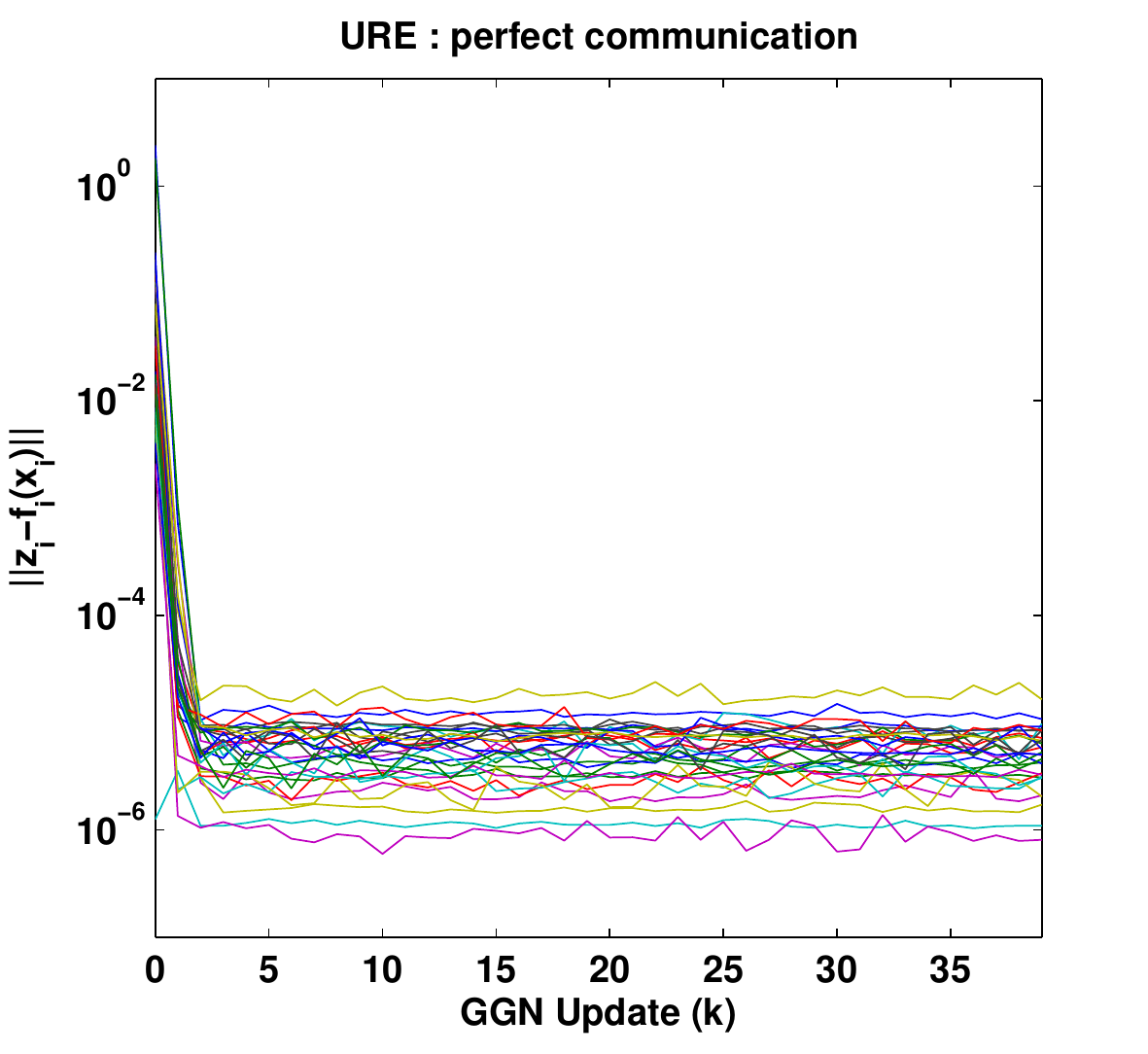}}}}
{\subfigure[][$\mathrm{MSE}_V$ Comparison]{\resizebox{0.32\textwidth}{!}{\includegraphics{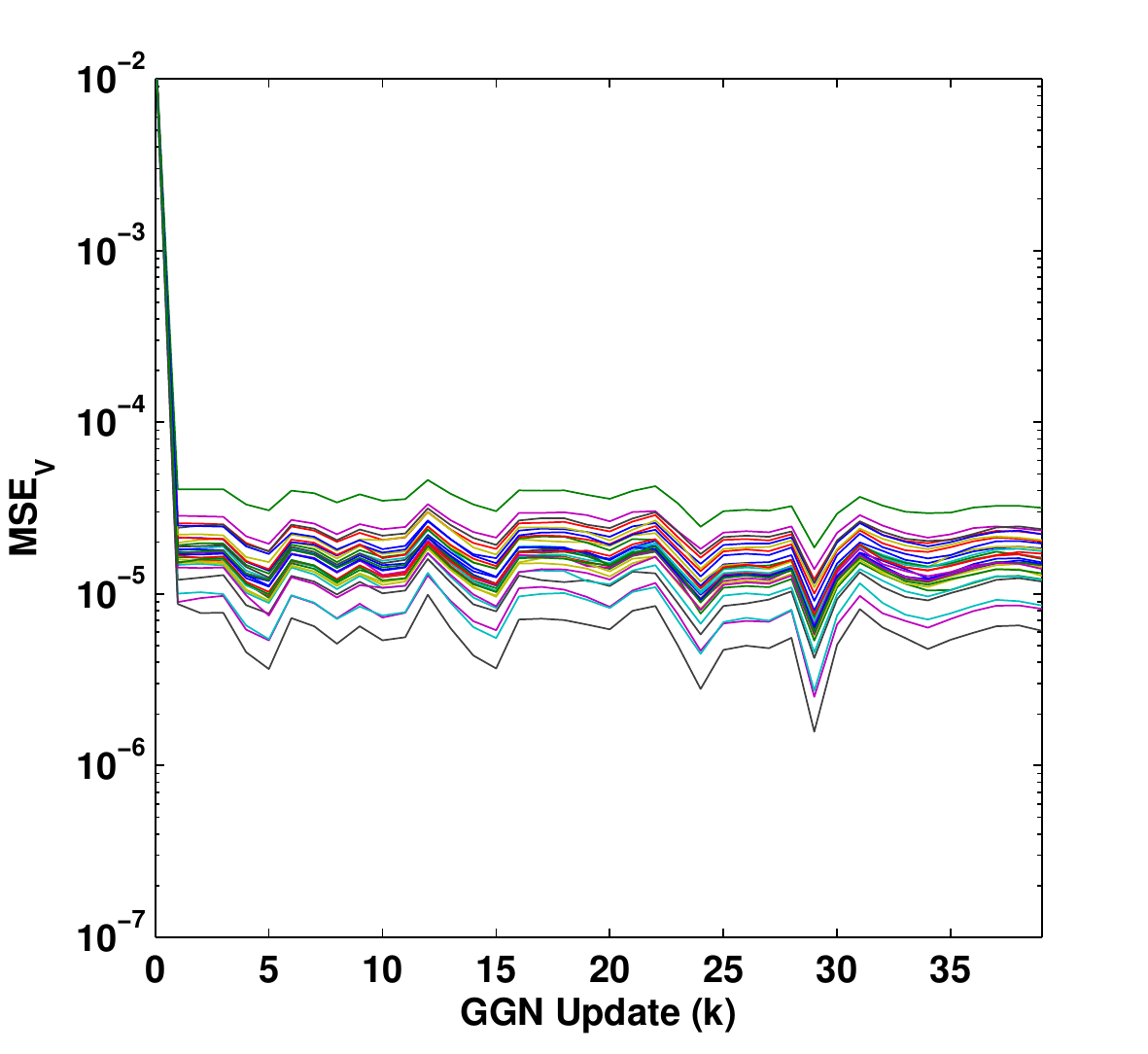}}}}
{\subfigure[][$\mathrm{MSE}_\Theta$ Comparison]{\resizebox{0.32\textwidth}{!}{\includegraphics{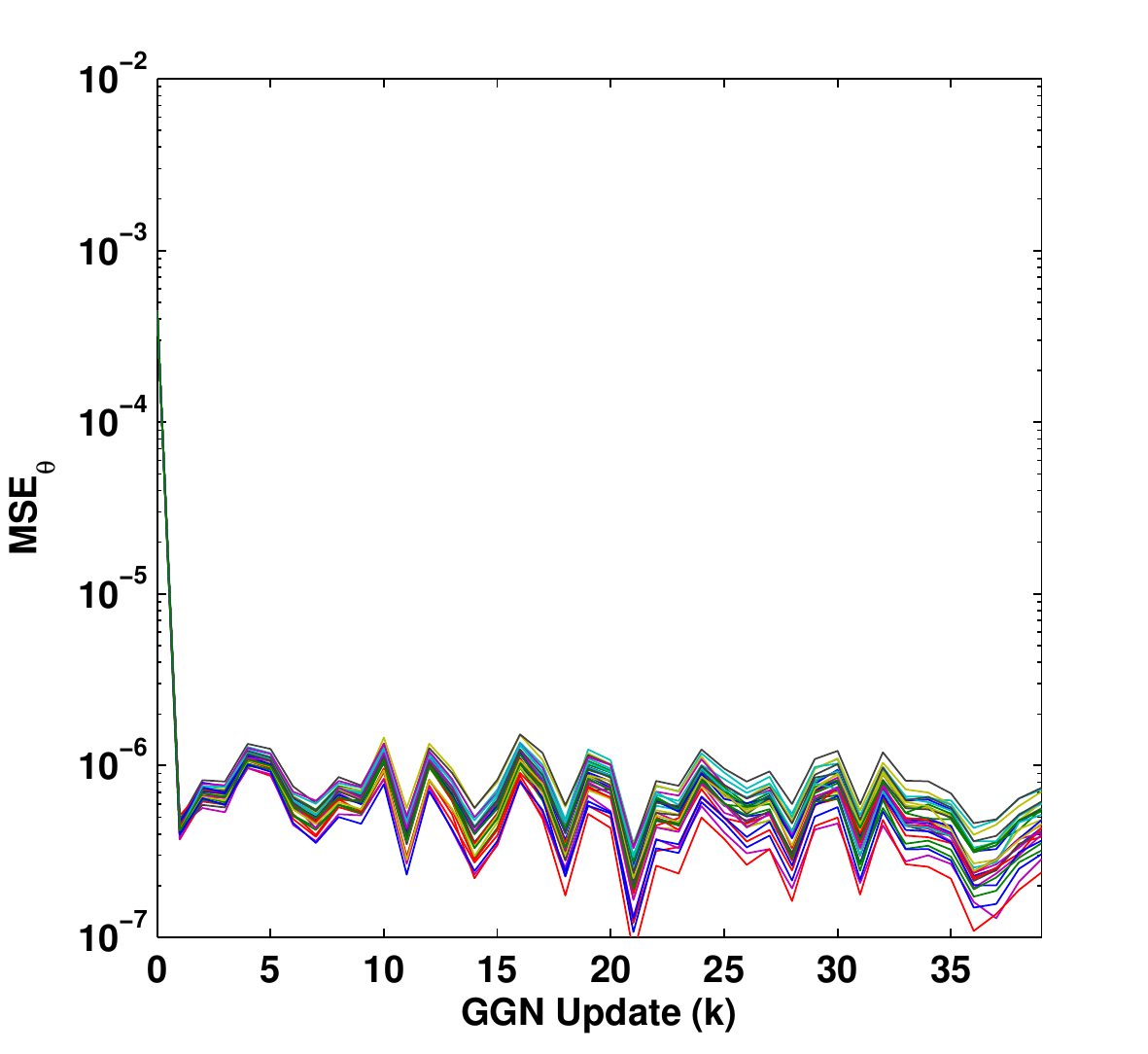}}}}\\
\vspace{-0.3cm}
{\subfigure[][Objective value]{\resizebox{0.32\textwidth}{!}{\includegraphics{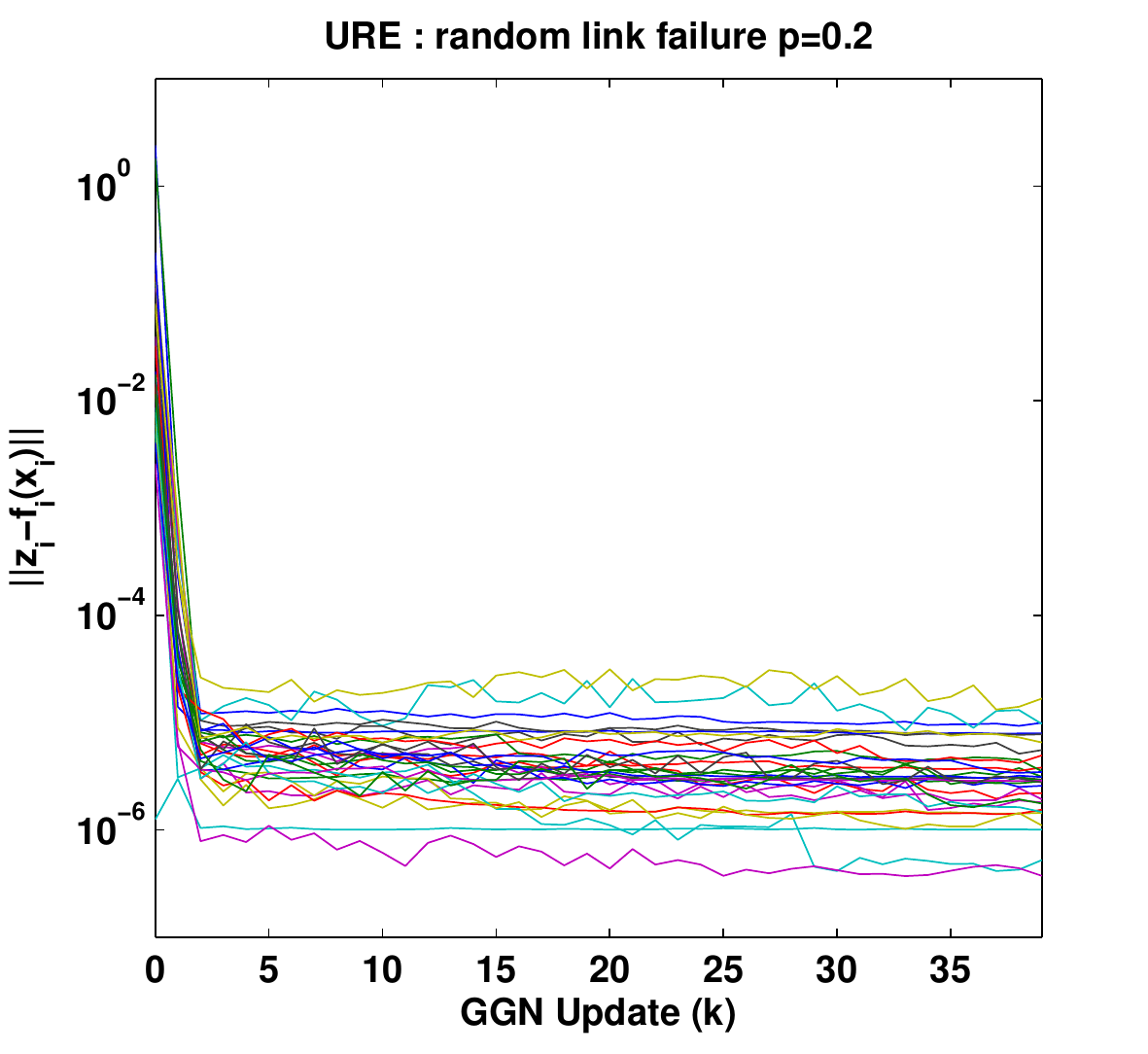}}}}
{\subfigure[][$\mathrm{MSE}_V$ Comparison]{\resizebox{0.32\textwidth}{!}{\includegraphics{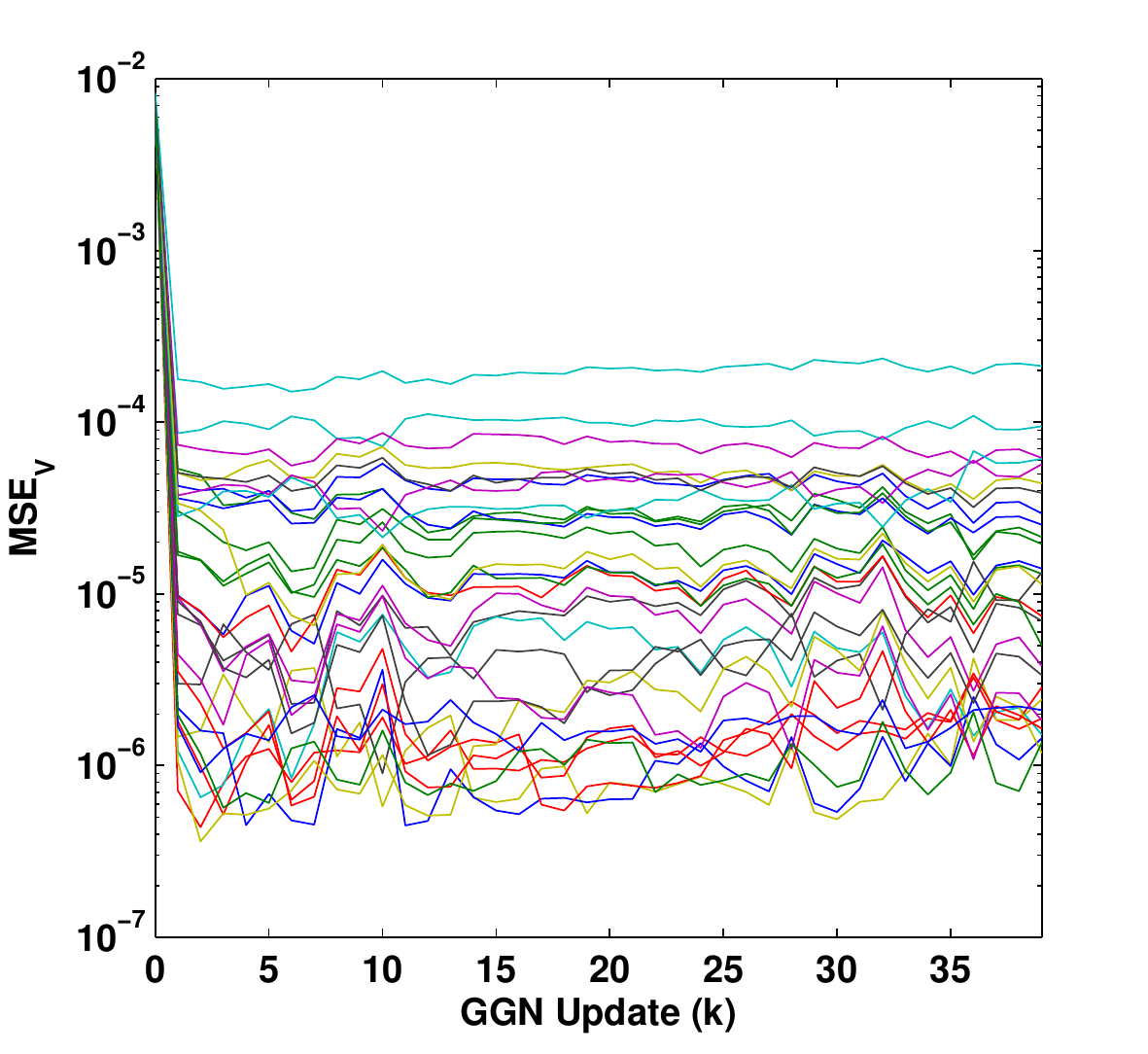}}}}
{\subfigure[][$\mathrm{MSE}_\Theta$ Comparison]{\resizebox{0.32\textwidth}{!}{\includegraphics{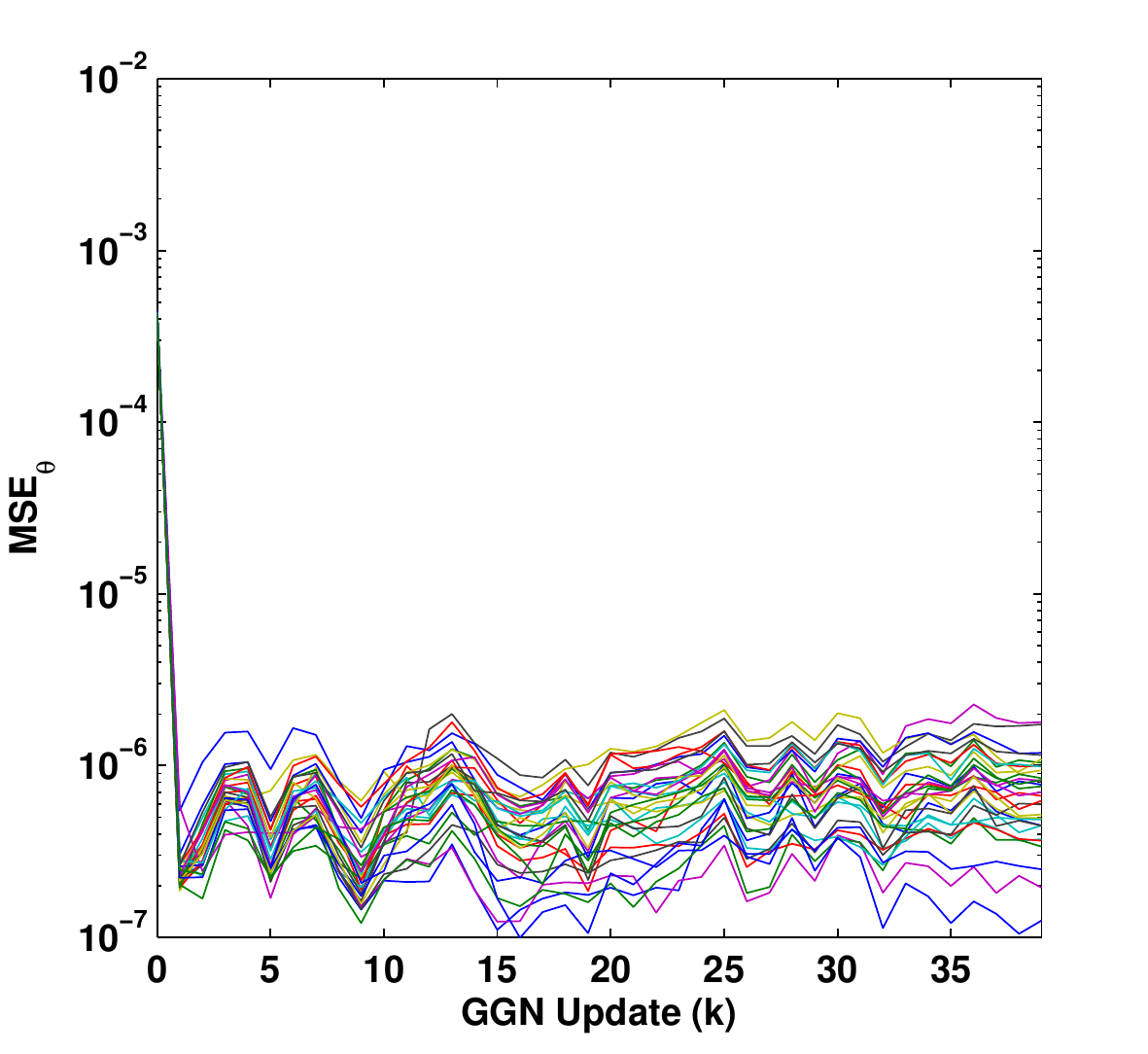}}}}
\vspace{0.2cm}
\end{center}
\vspace{-0.5cm}
\caption{MSE performance of GGN (URE Protocol) in IEEE-30 bus system with $I=N=30$ agents and $\mathcal{O}(N)$ pair-wise gossip exchanges. (top) : perfect communication (bottom) : $p=0.3$ random link failures (each line corresponds to one agent). }
\label{fig.GGN_jerf_IEEE30}
\end{figure*}

Clearly, the GGN algorithm converges much faster since it reaches the steady state error after $k =10$ updates (i.e., $k\ell_{\min} = 30$ exchanges). It is observed in Fig. \ref{fig.diff_obj1} to \ref{fig.diff_grad} that although the gossip exchange per update $\ell_{\min}=3$ is small and does not satisfy Assumption \ref{ell_ass}, both $\mathrm{Val}_k$ and the $\mathrm{Grad}_k$ still decrease exponentially as the iterations progress. On the other hand, the objective value and the gradient norm of the diffusion algorithm in \cite{xie2012fully} decrease slowly. Furthermore, the update of diffusion algorithms exhibit more fluctuations especially in the beginning, while the GGN algorithm conditions the gradient by the GN Hessian and therefore the update tends to be smooth and continues to lie in the proximity of the desired solution with high accuracy. Furthermore, the performances of the diffusion algorithm are sensitive to the step-size $\alpha_{\textrm{diff},\ell}$, since $\alpha_{\textrm{diff},\ell}=0.01\ell^{-1}$ is better initially due to less fluctuations as a result of the small step-size, while $0.3\ell^{-1}$ gradually outperforms $0.01\ell^{-1}$ due to the progress made by the larger step-size. However, when the step-size continues to increase, the performance starts to deteriorate from $\alpha_{\textrm{diff},\ell}=0.5\ell^{-1}$ to $\ell^{-1}$, and even diverges beyond a certain value.

\subsubsection{Estimation via Adaptive Processing}
Here we show numerically the applicability of the GGN algorithm to adaptive processing as described in \eqref{adaptive} against the method proposed in \cite{kar2008distributed} with the same network setting and step-sizes. Furthermore, we compare the global MSE performance of the GGN algorithm against the diffusion algorithm, given by
\begin{align}
	\mathrm{MSE}_V^{(k)} &= \frac{1}{I}\sum_{i=1}^I  \mathrm{MSE}_{V, i}^{(k)},~
	\mathrm{MSE}_\Theta^{(k)} = \frac{1}{I}\sum_{i=1}^I  \mathrm{MSE}_{\Theta,i}^{(k)}.
\end{align}

We generate $3$ snapshots of measurements $\{\mathbf{z}_i[t]\}_{i=1}^I$ for $t=1,\cdots,3$ based on the same state $\bar{\mathbf{x}}[t]=\bar{\mathbf{x}}$ by adding independent Gaussian noise with variance $\sigma^2=10^{-6}$, similar to the adaptive setting considered in \cite{kar2008distributed}. More specifically, we use $\ell_{\min}=3$ gossip exchanges between every two algorithm updates until $k=10$, thus leading to a total number of $30$ exchanges per snapshot. It can be seen from Fig. \ref{fig.diff_obj_dynamic} to \ref{fig.diff_Va_dynamic} that the proposed GGN algorithm tracks the state accurately when new measurements stream in, where the spikes observed in the plots are caused by the new measurements. Since the number of gossip exchanges is limited, the diffusion algorithm in \cite{xie2012fully} and \cite{kar2008distributed} convergence slowly and fail to track the state.

\subsection{MSE Performance under URE Protocol with Link Failures}
In this section, we examine the MSE performance of the GGN algorithm under the URE protocol with a fixed number of algorithm updates $K=40$. The performance is evaluated with a demanding setting, where we divide the $N$-bus system into $N$ sites and each site only communicates with one of its neighbors $10$ times on average. The network-wide communication volume in this scenario is on the order of the network diameter $\mathcal{O}(N)$, which implies the number of transmissions in the centralized scheme as if the local measurements are relayed and routed through the entire network. For simplicity, we simulate that at each exchange, the $i$-th distributed agents wakes up with uniform probability $1/I$ and picks a neighbor with equal probability $1/I$.

In order to show the robustness of the proposed algorithm, we examine the performance of the GGN algorithm for cases with random link failures, where any established link $\{i,j\}\in\mathcal{M}$ fails with probability $p=0.3$ independently. It is clear that this communication model with link failures may not satisfy Assumption \ref{connectivity_frequency}, but it is shown below that our approach is robust to the random setting and degrades gracefully with the probability of failures. In Fig. \ref{fig.GGN_jerf_IEEE30}, we track both the individual objective $\mathrm{Val}_i^{(k)}=\|\mathbf{z}_i-\mathbf{f}_i(\mathbf{x}_i^k)\|^2$ as well as the individual $\mathrm{MSE}_{V,i}^{(k)}$ and $\mathrm{MSE}_{\Theta,i}^{(k)}$ defined in \eqref{MSE_V} and \eqref{MSE_Theta}. It can be observed from the figure that the MSE curves of state estimates of different sites are highly consistent and they all converge asymptotically when there is no link failures. Similar behaviors can be observed for the case with random link failures, where the local estimate at each site is not in perfect consistence with the others, but the accuracy remains satisfactory compared to the perfect case and degrades gracefully with the probability of link failures. 

\section{Conclusions}
In this paper, we study the convergence and performance of the GGN algorithm and discuss its application in power system state estimation. The numerical results suggest that  that the proposed algorithm leads to accurate state estimates across the distributed areas,  is robust to link/node failures, with polinomial communication and computation cost.

\section{Acknowledgements}
We wish to thank the Associate Editor the anonymous Reviewers for their comments. Their suggestions helped
significantly in improving this article.

\appendices
\section{Proof of Lemma \ref{lem_error_recursion}}\label{proof_lem_error_recursion}
	To study the convergence of the GGN algorithm, we examine the update in \eqref{local_estimate} and re-write it with respect to the exact descent $\mathbf{d}_i^k$ in \eqref{exact_local_descent}
\begin{align}
	\mathbf{x}_i^{k+1} = P_{\mathbb{X}}\left[\mathbf{x}_i^k -\alpha\mathbf{d}_i^k + \alpha\left(\mathbf{d}_i^k-\mathbf{d}_i^k(\ell_k)\right)\right].
\end{align}
By subtracting the fixed point $\mathbf{x}^\star$ and using the non-expansive property of the operator $P_{\mathbb{X}}(\cdot)$ on the closed convex set $\mathbb{X}$, we have the following recursion
\begin{align*}
	\left\|\mathbf{x}_i^{k+1}-\mathbf{x}^\star\right\| \leq \left\|\mathbf{x}_i^k-\mathbf{x}^\star -\alpha \mathbf{d}_i^k\right\| +\alpha \left\|\mathbf{d}_i^k(\ell_k)-\mathbf{d}_i^k\right\|.
\end{align*}
For convenience, we denote $\mathbf{G}^\dagger(\cdot)$ as the pseudo-inverse of $ \mathbf{G}(\cdot)$. For any fixed point $\mathbf{x}^\star\in\mathbb{X}$ in \eqref{fixed_joint} such that $\mathbf{G}^\dagger (\mathbf{x}^\star)\mathbf{g}(\mathbf{x}^\star)=\mathbf{0}$, the first term can be equivalently written by substituting \eqref{decentralized_descent} as follows
\begin{align}\label{orig_update}
	\mathbf{x}_i^k-\mathbf{x}^\star - \alpha\mathbf{d}_i^k
	&=\mathbf{x}_i^k-\mathbf{x}^\star \\
	&~~~ - \alpha\mathbf{G}^\dagger (\mathbf{x}_i^k)\mathbf{g}(\mathbf{x}_i^k)
	 + \alpha\mathbf{G}^\dagger (\mathbf{x}^\star)\mathbf{g}(\mathbf{x}^\star).\nonumber
\end{align}
Using \eqref{central_descent} together with the invertibility condition of $\mathbf{G}(\mathbf{x})$ over $\mathbf{x}\in\mathbb{X}$ in Assumption \ref{lipshitz}, we have
\begin{align}\label{dummy_x}
    \mathbf{x}_i^k-\mathbf{x}^\star = \mathbf{G}^\dagger (\mathbf{x}_i^k)\mathbf{G} (\mathbf{x}_i^k)\left(\mathbf{x}_i^k-\mathbf{x}^\star\right).
\end{align}
Then by substituting \eqref{dummy_x} into \eqref{orig_update}, and meanwhile adding and subtracting simultaneously a term $\alpha\mathbf{G}^\dagger (\mathbf{x}_i^k)\mathbf{g}(\mathbf{x}^\star)$, we have the following expression
\begin{align}\label{recursion_term1}
	&\mathbf{x}_i^k-\mathbf{x}^\star - \alpha\mathbf{d}_i^k\\
	&=
	\mathbf{G}^\dagger (\mathbf{x}_i^k)\left[\mathbf{G}(\mathbf{x}_i^k)\left(\mathbf{x}_i^k-\mathbf{x}^\star\right) - \alpha\mathbf{g}(\mathbf{x}_i^k) + \alpha\mathbf{g}(\mathbf{x}^\star)\right]\\
	&~~~ + \alpha\left[ \mathbf{G}^\dagger (\mathbf{x}^\star) - \mathbf{G}^\dagger (\mathbf{x}_i^k)\right]\mathbf{g}(\mathbf{x}^\star).
\end{align}
The expression in the first term can be re-written with the mean-value theorem as follows
\begin{align}
	&\alpha\mathbf{g}(\mathbf{x}^\star)-\alpha\mathbf{g}(\mathbf{x}) - \mathbf{G}(\mathbf{x})(\mathbf{x}^\star-\mathbf{x})\\
	&=\alpha\left[\int_{0}^1\mathbf{G}(\mathbf{x} + t(\mathbf{x}^\star-\mathbf{x}))(\mathbf{x}^\star-\mathbf{x}) \mathrm{d}t\right] - \mathbf{G}(\mathbf{x})(\mathbf{x}^\star-\mathbf{x})\nonumber\\
	&=\alpha\left(\int_{0}^1\left[\mathbf{G}(\mathbf{x} + t(\mathbf{x}^\star-\mathbf{x}))- \mathbf{G}(\mathbf{x})\right](\mathbf{x}^\star-\mathbf{x}) \mathrm{d}t\right) \nonumber\\
	&~~~ - (1-\alpha) \mathbf{G}(\mathbf{x})(\mathbf{x}^\star-\mathbf{x}),\nonumber
\end{align}
whose norm can be bounded by using Assumption \ref{lipshitz} as
\begin{align}
	&\left\|\alpha\mathbf{g}(\mathbf{x}^\star)-\alpha\mathbf{g}(\mathbf{x}) - \mathbf{G}(\mathbf{x})(\mathbf{x}^\star-\mathbf{x})\right\|\\
	&\leq
	\alpha\left[\int_{0}^1\left\|\mathbf{G}(\mathbf{x} + t(\mathbf{x}^\star-\mathbf{x}))- \mathbf{G}(\mathbf{x})\right\|\mathrm{d}t\right]\left\|\mathbf{x}-\mathbf{x}^\star\right\|\nonumber\\
	&~~~ + (1-\alpha)\sigma_{\max}\left\|\mathbf{x}-\mathbf{x}^\star\right\|.\nonumber
\end{align}
From the Lipschitz condition in Assumption \ref{lipshitz}, we have
\begin{align*}
	\int_{0}^1\left\|\mathbf{G}(\mathbf{x} + t(\mathbf{x}^\star-\mathbf{x}))- \mathbf{G}(\mathbf{x})\right\|\mathrm{d}t
	\leq \omega \left\|\mathbf{x}-\mathbf{x}^\star\right\| \int_{0}^1 t\mathrm{d}t.
\end{align*}
Thus, if condition (3) of Assumption \ref{lipshitz} holds, we have
\begin{align}\label{term1}
	&\left\|\alpha\mathbf{g}(\mathbf{x}^\star)-\alpha\mathbf{g}(\mathbf{x}) - \mathbf{G}(\mathbf{x})(\mathbf{x}^\star-\mathbf{x})\right\|\\
	&\leq \frac{\alpha \omega}{2}\left\|\mathbf{x}-\mathbf{x}^\star\right\|^2+ (1-\alpha)\sigma_{\max}\left\|\mathbf{x}-\mathbf{x}^\star\right\|,\nonumber
\end{align}
and finally according to \cite[Lemma 1]{salzo2011convergence}, we have 
\begin{align}\label{term2}
	\|\mathbf{G}^{\dagger} (\mathbf{x})-\mathbf{G}^{\dagger} (\mathbf{x}^\star)\|
	& \leq \sqrt{2}\|\mathbf{G}^{\dagger} (\mathbf{x})\|\|\mathbf{G}^{\dagger} (\mathbf{x}^\star)\|\|\mathbf{G}(\mathbf{x})-\mathbf{G}(\mathbf{x}^\star)\|\nonumber\\
	&\leq \frac{\sqrt{2}\omega}{\sigma_{\min}^2}\|\mathbf{x}-\mathbf{x}^\star\|.
\end{align}
By definition we have $\|\mathbf{G}^\dagger(\boldsymbol{x})\|^2 =\|\left(\mathbf{G}^T(\boldsymbol{x})\mathbf{G}(\boldsymbol{x})\right)^{-1}\|$. Also Assumption \ref{lipshitz} implies $\|\mathbf{G}^\dagger(\boldsymbol{x})\|\leq 1/\sigma_{\min}$.

For convenience, we let $\epsilon_{\min}\triangleq\left\|\mathbf{g}(\mathbf{x}^\star)\right\|$ be the goodness of fit at $\mathbf{x}^\star$ and define the following constants
\begin{align}
    T_1 &\triangleq \frac{\alpha\omega}{2\sigma_{\min}},\quad
    T_2 \triangleq (1-\alpha)\frac{\sigma_{\max}}{\sigma_{\min}}+\frac{\sqrt{2}\alpha\omega\epsilon_{\min}}{\sigma_{\min}^2}.
\end{align}
Then, substituting $\|\mathbf{G}^\dagger (\mathbf{x})\|\leq 1/\sigma_{\min}$ and \eqref{term1}\eqref{term2} back to \eqref{recursion_term1} and using \eqref{T_constants}, we have
\begin{align*}
	\left\|\mathbf{x}_i^k-\mathbf{x}^\star - \alpha\mathbf{d}_i^k\right\|
	&\leq
	T_1
    \left\|\mathbf{x}_i^k-\mathbf{x}^\star\right\|^2
	+
    T_2\left\|\mathbf{x}_i^k-\mathbf{x}^\star\right\|,
\end{align*}
Therefore, we have the error recursion \eqref{distributed_error_recursion}.

\section{Proof of Theorem \ref{theorem_convergence}}\label{proof_theorem_convergence}
If the discrepancy error is upper bounded by a constant $\kappa\geq 0$ such that $\|\mathbf{d}_i^k(\ell_k)-\mathbf{d}_i^k\| \leq \kappa$, then from Lemma \ref{lem_error_recursion}, the recursion can be simplified as
\begin{align}
	\left\|\mathbf{x}_i^{k+1}-\mathbf{x}^\star\right\|
	&\leq
	T_1 \left\|\mathbf{x}_i^k-\mathbf{x}^\star\right\|^2
	+T_2\left\|\mathbf{x}_i^k-\mathbf{x}^\star\right\|
	 + \alpha\kappa.
\end{align}
where $T_1$ and $T_2$ are given in \eqref{T_constants}. Let $\zeta_{i,k}=\left\|\mathbf{x}_i^k-\mathbf{x}^\star\right\|$, then the error recursion can be expressed as a dynamical system as
\begin{align}
	\zeta_{i,k+1} \leq T_1 \zeta_{i,k}^2 + T_2 \zeta_{i,k} + \alpha\kappa, \quad \zeta_{i,k}>0.
\end{align}
Since $\zeta_{i,k}$ is non-negative, this error dynamic can be upper bounded by the dynamical system of $\rho_{k+1}=\psi(\rho_k)$ with
\begin{align}
	\psi(\rho_k)= T_1 \rho_k^2 + T_2 \rho_k + \alpha\kappa, \quad \rho_k>0,
\end{align}
whose equilibrium points are obtained by solving
\begin{align}\label{contraction}
	\rho  = T_1 \rho^2 + T_2 \rho +\alpha \kappa.
\end{align}
When $\kappa$ satisfies
\begin{align}\label{Delta_eq}
    (1-T_2)^2 - 4\alpha T_1\kappa\geq 0,
\end{align}
the equilibrium points of \eqref{contraction} exist and are obtained as  \eqref{rho_def}.

Now let $\dot{\psi}(\rho)\triangleq {\mathrm{d}\psi(\rho)}/{\mathrm{d}\rho}$ be the first order derivative of the dynamics. According to \cite[cf. Proposition 1.9]{galor2007discrete}, an equilibrium point is a stable sink if $|\dot{\psi}(\cdot)|<1$ and unstable otherwise. Thus, the equilibrium point $\rho_{\max}$ is unstable since the following is always true
\begin{align}
    \left|\dot{\psi}(\rho_{\max})\right| &= |2T_1 \rho_{\max} +T_2|\\
    &=\left|1+\sqrt{(1-T_2)^2-4\alpha T_1\kappa}\right|>1,
\end{align}
while the point $\rho_{\min}$ is a sink if
\begin{align}
    \left|\dot{\psi}(\rho_{\min})\right|
    &=\left|1-\sqrt{(1-T_2)^2-4\alpha T_1\kappa}\right|<1.
\end{align}
To guarantee $\left|\dot{\psi}(\rho_{\min})\right|<1$, it requires
\begin{align}
    0 < (1-T_2)^2 - 4\alpha T_1\kappa < 4,
\end{align}
which together with \eqref{Delta_eq} leads to the following condition on the bounded perturbation $\kappa$
\begin{align}
    \frac{T_2^2-2T_2-3}{4\alpha T_1}< \kappa <\frac{T_2^2-2T_2+1}{4\alpha T_1}.
\end{align}
Clearly, given an arbitrary $\alpha\in(0,1]$, the lower bound on $\kappa$ is unrealistic if $T_2^2-2T_2-3>0$ since the lower bound could approach infinity as $\alpha\rightarrow 0$. Therefore, to ensure convergence with an arbitrarily small perturbation, it is sufficient to have
\begin{align}
    T_2^2-2T_2-3<0 \Longrightarrow -1<T_2<3.
\end{align}
Since $T_2\geq 0$ by definition \eqref{T_constants}, the condition becomes
\begin{align}
    0\leq(1-\alpha)\frac{\sigma_{\max}}{\sigma_{\min}}+\frac{\sqrt{2}\alpha\omega\epsilon_{\min}}{\sigma_{\min}^2}<3.
\end{align}
By re-arranging the terms, this condition is equivalent to
\begin{align}
    &
    \begin{cases}
       \displaystyle \frac{\sqrt{2}\alpha\omega\epsilon_{\min}}{\sigma_{\min}^2}
       <3-(1-\alpha)\frac{\sigma_{\max}}{\sigma_{\min}}\\
       \displaystyle 3-(1-\alpha)\frac{\sigma_{\max}}{\sigma_{\min}}>0
    \end{cases},
\end{align}
which can be simplified as
\begin{align}
    \begin{cases}
        \displaystyle \omega\epsilon_{\min} < \frac{\sigma_{\min}^2}{\sqrt{2}\alpha}\left[ 3-(1-\alpha)\frac{\sigma_{\max}}{\sigma_{\min}}\right]\\
        \displaystyle \max\left\{1-\frac{3\sigma_{\min}}{\sigma_{\max}},0\right\}<\alpha\leq 1
    \end{cases}.
\end{align}
Thus, if the initial error $\zeta_{i,0}>\rho_{\max}$, the error keeps growing. On the other hand, if the errors are bounded by $0<\zeta_{i,k}<\rho_{\max}$ for all $i$'s and $k$'s, the algorithm reaches the equilibrium error floor $\rho_{\min}$. Thus, as long as the initialization error $\zeta_{i,0}$ satisfies $0< \zeta_{i,0} < \rho_{\max}$ for $i=1,\cdots,I$, the algorithm progresses with contracting error until reaching the error floor $\rho_{\min}$ due to the constant bounded perturbation $\kappa$. 

As a result, as long as the initial condition $\mathbf{x}_i^0$ satisfies $\left\|\mathbf{x}_i^0-\mathbf{x}^\star\right\|<\rho_{\max}$ with respect to a certain fixed point $\mathbf{x}^\star$, the error norm is upper bounded by
\begin{align*}
	\underset{k\rightarrow\infty}{\limsup} \left\|\mathbf{x}_i^k-\mathbf{x}^\star\right\| \leq \rho_{\min}.
\end{align*}
Instead, if $\left\|\mathbf{x}_i^0-\mathbf{x}^\star\right\|>\rho_{\max}$, the error grows without bound.

\section{Proof of Lemma \ref{lem_gossip_error}}\label{proof_lem_gossip_error}
Using \eqref{gossip_exchange_matrix}, we evaluate the deviation of $\bdsb{\mathcal{H}}_k(\ell)$ from the average $\bar{\bdsb{\mathcal{H}}}_k = \left[\mathbf{1}^T\otimes\mathbf{I}_{N_{\mathcal{H}}}\right]\bdsb{\mathcal{H}}_k(0)/I$ for a finite $\ell$. By subtracting the average $\bar{\bdsb{\mathcal{H}}}_k$ on both sides of \eqref{gossip_exchange_matrix}, we have
\begin{align*}
	&\bdsb{\mathcal{H}}_k(\ell) - \bar{\bdsb{\mathcal{H}}}_k\\
	&=
	 \left[\mathbf{W}_k(\ell)\otimes \mathbf{I}_{N_{\mathcal{H}}}\right] \bdsb{\mathcal{H}}_k(\ell-1) - \frac{\mathbf{1}\mathbf{1}^T\otimes\mathbf{I}_{N_{\mathcal{H}}}}{I} \bdsb{\mathcal{H}}_k(0)\\
	&=
	\left[\prod_{\ell'=0}^{\ell}\mathbf{W}_k(\ell')\otimes \mathbf{I}_{N_{\mathcal{H}}}\right] \bdsb{\mathcal{H}}_k(0) - \frac{\mathbf{1}\mathbf{1}^T\otimes\mathbf{I}_{N_{\mathcal{H}}}}{I} \bdsb{\mathcal{H}}_k(0)\\
	&=
	\left[\left(\prod_{\ell'=0}^{\ell}\mathbf{W}_k(\ell') - \frac{\mathbf{1}\mathbf{1}^T}{I}\right) \otimes \mathbf{I}_{N_{\mathcal{H}}}\right]  \bdsb{\mathcal{H}}_k(0).
\end{align*}
Then, we bound the norms of the above equation as
\begin{align}
	\left\|\bdsb{\mathcal{H}}_k(\ell) - \bar{\bdsb{\mathcal{H}}}_k\right\|
	&\leq \left\|\prod_{\ell'=0}^{\ell}\mathbf{W}_k(\ell') - \frac{\mathbf{1}\mathbf{1}^T}{I}\right\|\left\|\bdsb{\mathcal{H}}_k(0)\right\|.
\end{align}
Using Lemma \ref{weight_matrix_jroperty} and the norm inequality $\left\|\cdot\right\|\leq \left\|\cdot\right\|_F$, we have
\begin{align*}
	\left\|\bdsb{\mathcal{H}}_k(\ell) - \bar{\bdsb{\mathcal{H}}}_k\right\|
	&\leq \left\|\prod_{\ell'=0}^{\ell}\mathbf{W}_k(\ell') - \frac{\mathbf{1}\mathbf{1}^T}{I}\right\|_F\left\|\bdsb{\mathcal{H}}_k(0)\right\|\\
	&\leq \left[2I\left(\frac{1+\eta^{-L_0}}{1-\eta^{L_0}}\right)\lambda_\eta^\ell\right]\left\|\bdsb{\mathcal{H}}_k(0)\right\|.
\end{align*}
The quantity $\left\|\bdsb{\mathcal{H}}_k(0)\right\|$ is by definition \eqref{information_vec} determined as
\begin{align}
	\left\|\bdsb{\mathcal{H}}_k(0)\right\|^2
	&=
	\sum_{i=1}^I \left\| \mathbf{h}_{k,i}(0)\right\|^2 + \sum_{i=1}^I\left\| \mathbf{H}_{k,i}(0) \right\|_F^2\nonumber\\
	&=
	\sum_{i=1}^I \left(\left\| \mathbf{G}_i^T(\mathbf{x}_i^k)\mathbf{g}_i(\mathbf{x}_i^k) \right\|^2 +
	\left\| \mathbf{G}_i^T(\mathbf{x}_i^k)\mathbf{G}_i(\mathbf{x}_i^k) \right\|_F^2\right)\nonumber\\
	&\leq
	I\sigma_{\max}^2( \epsilon_{\max} ^2 + N\sigma_{\max}^2),
\end{align}
where the norm inequality is used
\begin{align*}
	\| \mathbf{G}_i^T(\mathbf{x})\mathbf{G}_i(\mathbf{x}) \|_F^2
	&\leq N\| \mathbf{G}^T(\mathbf{x})\mathbf{G}(\mathbf{x}) \|_2^2= N\sigma_{\max}^4.
\end{align*}
Letting $C=2I\sqrt{I\sigma_{\max}^2( \epsilon_{\max} ^2 + N\sigma_{\max}^2)}\left(\frac{1+\eta^{-L_0}}{1-\eta^{L_0}}\right)$, then the error is bounded as $\left\|\bdsb{\mathcal{H}}_k(\ell) -  \bar{\bdsb{\mathcal{H}}}_k\right\|
	\leq C\lambda_\eta^{\ell}$.

By definition of $\mathbf{e}_{k,i}(\ell)$ and $\mathbf{E}_{k,i}(\ell)$, we have
\begin{align}
	\bdsb{\mathcal{H}}_k(\ell) - \bar{\bdsb{\mathcal{H}}}_k
	=
	\begin{bmatrix}
		\mathbf{e}_{k,1}(\ell)\\
		\mathrm{vec}\left[\mathbf{E}_{k,1}(\ell)\right]\\
		\vdots\\
		\mathbf{e}_{k,I}(\ell)\\
		\mathrm{vec}\left[\mathbf{E}_{k,I}(\ell)\right]
	\end{bmatrix},
\end{align}
and hence the norm of each component is bounded by the total norm $\left\|\mathbf{e}_k(\ell)\right\| < C\lambda_\eta^{\ell}$ and $\left\|\mathbf{E}_k(\ell)\right\|_F < C\lambda_\eta^{\ell}$.

\section{Proof of Lemma \ref{lem_disagreement}}\label{proof_lem_disagreement}
We prove this result by mathematical induction. We will repetitively use matrix expansion \cite{horntopics} for any $\mathbf{Z}$ and $\delta\mathbf{Z}$,
\begin{align}\label{expansion}
	 (\mathbf{Z}+\delta\mathbf{Z})^{-1}=\sum_{q=0}^{\infty}(-1)^q\left(\mathbf{Z}^{-1}\delta\mathbf{Z}\right)^q\mathbf{Z}^{-1}
\end{align}
as long as $\left\|\mathbf{Z}^{-1}\delta\mathbf{Z}\right\|<1$

\subsection{Initial Case: $k=1$}
Given $\mathbf{x}_i^0=\mathbf{x}^0$ for all $i$, then for any $i\neq j$ we have
\begin{align}\label{upperbound_disagreement_0}
	\left\|\mathbf{x}_i^1-\mathbf{x}_j^1\right\| \leq
	\left\|\mathbf{d}_i^0(\ell_0)-\mathbf{d}_j^0(\ell_0)\right\|,
\end{align}
where the discrepancy is expressed explicitly as
\begin{align}
	\mathbf{d}_i^0(\ell_0)-\mathbf{d}_j^0(\ell_0)
	&=
	\left[\bar{\mathbf{H}}_0 +\mathbf{E}_{0,i}(\ell_0) \right]^{-1}\left[\bar{\mathbf{h}}_0 +  \mathbf{e}_{0,i}(\ell_0)\right]\nonumber\\
	&-\left[\bar{\mathbf{H}}_0 +\mathbf{E}_{0,j}(\ell_0) \right]^{-1}\left[\bar{\mathbf{h}}_0 +  \mathbf{e}_{0,j}(\ell_0)\right].
\end{align}
Thus, if $\mathbf{E}_{0,i}(\ell_0)$, $\mathbf{E}_{0,j}(\ell_0)$ are small enough, the expansion in \eqref{expansion} can be applied here to simplify the expression.

\subsubsection{Matrix series expansion}
Since $\mathbf{x}_i^0=\mathbf{x}^0$ for all $i$ such that $\bar{\mathbf{H}}_0=\mathbf{Q}(\mathbf{x}_i^0) $ and $\bar{\mathbf{h}}_0=\mathbf{q}(\mathbf{x}_i^0)$, they can be bounded based on Assumption \ref{lipshitz} as follows
\begin{align}\label{h_H_0_bound}
	\left\|\bar{\mathbf{h}}_0\right\| &= \left\|\mathbf{q}(\mathbf{x}_i^0)\right\|\\
                                    &=\frac{1}{I}\left\|\mathbf{G}^T(\mathbf{x}_i^0)\mathbf{g}(\mathbf{x}_i^0)\right\| \leq \frac{\sigma_{\max}  \epsilon_{\max} }{I},\\
	\left\|\bar{\mathbf{H}}_0^{-1}\right\| &=\left\|\mathbf{Q}^{-1}(\mathbf{x}_i^0)\right\| \\
                                            &= I\left\|\left(\mathbf{G}^T(\mathbf{x}_i^0)\mathbf{G}(\mathbf{x}_i^0)\right)^{-1} \right\| \leq \frac{I}{\sigma_{\min}^2}.
\end{align}
Note that from the norm inequality of sub-matrices
\begin{align*}
	\left\|\bar{\mathbf{H}}_0^{-1} \mathbf{E}_{0,i}(\ell_0) \right\| &\leq \left\|\bar{\mathbf{H}}_0^{-1}\right\|\left\| \mathbf{E}_{0,i}(\ell_0) \right\|\leq \left\|\bar{\mathbf{H}}_0^{-1}\right\|\left\|\mathbf{E}_0(\ell_0)\right\|_F\\
	\left\|\bar{\mathbf{H}}_0^{-1} \mathbf{E}_{0,j}(\ell_0) \right\| &\leq \left\|\bar{\mathbf{H}}_0^{-1}\right\|\left\| \mathbf{E}_{0,j}(\ell_0) \right\|\leq \left\|\bar{\mathbf{H}}_0^{-1}\right\|\left\|\mathbf{E}_0(\ell_0)\right\|_F,
\end{align*}	
and by Assumption \ref{ell_ass} we have $\ell_0\geq \ell_{\min}$. From Lemma \ref{lem_gossip_error} and Assumption \ref{ell_ass}, the above inequalities can be bounded as
\begin{align}\label{ratio2} \left\|\bar{\mathbf{H}}_0^{-1}\right\|\left\|\mathbf{E}_0(\ell_0)\right\|_F
	&\leq \frac{I}{\sigma_{\min}^2} C \lambda_\eta^{\ell_0}
    =\lambda_\eta^{(\ell_0-\ell_{\min})}\frac{IC}{\sigma_{\min}^2}\lambda_\eta^{\ell_{\min}}.
\end{align}
Choosing $\ell_{\min}$ according to \eqref{denominator}, we have $\lambda_\eta^{(\ell_0-\ell_{\min})}<1$ and
\begin{align*}
    \ell_{\min}> \log\left(\frac{\xi}{4D}\right)/\log\lambda_\eta
    ~~
    \Longrightarrow
    ~~
    \frac{IC}{\sigma_{\min}^2}\lambda_\eta^{\ell_{\min}}
    < \frac{IC}{4\sigma_{\min}^2 D}\xi.
\end{align*}
For notation convenience, we define
\begin{align}\label{tildexi}
    \tilde{\xi} = \frac{IC}{\sigma_{\min}^2 D}\xi,
\end{align}
and clearly, we have $0<\tilde{\xi}<\xi<1/2$ according to the definition of $D$ in \eqref{denominator} by Assumption \ref{ell_ass}. Therefore, letting $\delta\mathbf{Z}=\mathbf{E}_0(\ell_0)$ and $\mathbf{Z}=\bar{\mathbf{H}}_0$, it follows from Lemma \ref{lem_gossip_error} that
\begin{align}\label{ratio2}
    \left\|\mathbf{Z}^{-1}\delta\mathbf{Z}\right\|=\left\|\bar{\mathbf{H}}_0^{-1}\right\|\left\|\mathbf{E}_0(\ell_0)\right\|_F
	&\leq \frac{1}{4}\lambda_\eta^{(\ell_0-\ell_{\min})} \tilde{\xi} < \frac{1}{8}.
\end{align}
and the expansion holds. By the matrix series expansion and grouping all the high order terms
$q\geq 1$, we have
\begin{align}\label{gossip_error_0}
	&\mathbf{d}_i^0(\ell_0)-\mathbf{d}_j^0(\ell_0)\\
	&=
    \left[\bar{\mathbf{H}}_0^{-1}-\sum_{q=1}^\infty (-1)^q\left(\bar{\mathbf{H}}_0^{-1}\mathbf{E}_{0,i}(\ell_0)\right)^q\bar{\mathbf{H}}_0^{-1}\right]
	\left[\bar{\mathbf{h}}_0+\mathbf{e}_{0,i}(\ell_0)\right]\nonumber\\
    &~~ -\left[\bar{\mathbf{H}}_0^{-1}-\sum_{q=1}^\infty (-1)^q\left(\bar{\mathbf{H}}_0^{-1}\mathbf{E}_{0,j}(\ell_0)\right)^q\bar{\mathbf{H}}_0^{-1}\right]
	\left[\bar{\mathbf{h}}_0+\mathbf{e}_{0,j}(\ell_0)\right].\nonumber
\end{align}
To simplify the above expression, we write it in three terms $\mathbf{D}_1(\ell_0)$, $\mathbf{D}_2(\ell_0)$ and $\mathbf{D}_3(\ell_0)$ as follows
\begin{align*}
	\mathbf{d}_i^0(\ell_0)-\mathbf{d}_j^0(\ell_0)
	&= \mathbf{D}_1(\ell_0) + \mathbf{D}_2(\ell_0) + \mathbf{D}_3(\ell_0),
\end{align*}
where $\mathbf{D}_1(\ell_0)\triangleq \bar{\mathbf{H}}_0^{-1}\left[\mathbf{e}_{0,i}(\ell_0)-\mathbf{e}_{0,j}(\ell_0)\right]$ and
\begin{align*}
    \mathbf{D}_2(\ell_0) &\triangleq \sum_{q=1}^\infty(-1)^q\left(\bar{\mathbf{H}}_0^{-1}\mathbf{E}_{0,j}(\ell_0)\right)^q\bar{\mathbf{H}}_0^{-1}\bar{\mathbf{h}}_0\\
    &~~~~-\sum_{q=1}^\infty(-1)^q\left(\bar{\mathbf{H}}_0^{-1}\mathbf{E}_{0,i}(\ell_0)\right)^q\bar{\mathbf{H}}_0^{-1}\bar{\mathbf{h}}_0\\
    \mathbf{D}_3(\ell_0) &\triangleq \sum_{q=1}^\infty(-1)^q\left(\bar{\mathbf{H}}_0^{-1}\mathbf{E}_{0,j}(\ell_0)\right)^q\bar{\mathbf{H}}_0^{-1}\mathbf{e}_{0,j}(\ell_0)\\
    &~~~~ -\sum_{q=1}^\infty(-1)^q\left(\bar{\mathbf{H}}_0^{-1}\mathbf{E}_{0,i}(\ell_0)\right)^q\bar{\mathbf{H}}_0^{-1}\mathbf{e}_{0,i}(\ell_0).
\end{align*}

\subsubsection{Proof of success when $k=1$}
According to the triangular inequality for norms, we can bound
\begin{align*}
	\left\|\mathbf{e}_{0,i}(\ell_0) - \mathbf{e}_{0,j}(\ell_0)\right\| &\leq 2\left\|\mathbf{e}_0(\ell_0)\right\|\\
	\left\|\mathbf{E}_{0,i}(\ell_0) - \mathbf{E}_{0,j}(\ell_0)\right\| &\leq 2\left\|\mathbf{E}_0(\ell_0)\right\|_F.
\end{align*}

Using \eqref{ratio2}, we can bound the norm of the first term as
\begin{align}\label{first_term}
    \left\|\mathbf{D}_1(\ell_0)\right\|
    &\leq 2\left\|\bar{\mathbf{H}}_0^{-1}\right\|\left\|\mathbf{e}_0(\ell_0)\right\|
    \leq \frac{1}{2}\tilde{\xi}\lambda_\eta^{(\ell_0-\ell_{\min})}.
\end{align}
Similarly, the infinite sum in the second term is bounded as
\begin{align}\label{second_term}
    &\left\|\sum_{q=1}^\infty(-1)^q\left[\left(\bar{\mathbf{H}}_0^{-1}\mathbf{E}_{0,j}(\ell_0)\right)^q-\left(\bar{\mathbf{H}}_0^{-1}\mathbf{E}_{0,i}(\ell_0)\right)^q\right]\right\|\nonumber\\
    &\leq
    2\sum_{q=1}^\infty \left(\left\|\bar{\mathbf{H}}_0^{-1}\right\|\left\|\mathbf{E}_0(\ell_0)\right\|_F\right)^q\nonumber\\
    &\leq 2\sum_{q=1}^\infty \left(\frac{1}{4}\tilde{\xi}\lambda_\eta^{(\ell_0-\ell_{\min})}\right)^q
    =
    \frac{1}{2}\frac{\tilde{\xi}\lambda_\eta^{(\ell_0-\ell_{\min})}}{\left(1-\frac{1}{4}\tilde{\xi}\lambda_\eta^{(\ell_0-\ell_{\min})}\right)},
\end{align}
where the last equality comes from the convergence of geometric series $\lim_{K\rightarrow\infty}\sum_{k=1}^{K}a^k=a/(1-a)$ for any $|a|<1$. Since $0<\tilde{\xi}<\xi<1/2$ and $\lambda_\eta^{(\ell_0-\ell_{\min})}\leq 1$, then
\begin{align}\label{trick}
    \frac{\tilde{\xi}\lambda_\eta^{(\ell_0-\ell_{\min})}}{\left(1-\frac{1}{4}\tilde{\xi}\lambda_\eta^{(\ell_0-\ell_{\min})}\right)}<2\tilde{\xi}\lambda_\eta^{(\ell_0-\ell_{\min})}
\end{align}
and thus the norm of the second term is bounded as
\begin{align*}
    \left\|\mathbf{D}_2(\ell_0)\right\|
    &\leq \tilde{\xi}\lambda_\eta^{(\ell_0-\ell_{\min})} \left\|\bar{\mathbf{H}}_0^{-1}\right\|\left\|\bar{\mathbf{h}}_0\right\|\\
    &\leq  \frac{\sigma_{\max} \epsilon_{\max} }{\sigma_{\min}^2}\tilde{\xi}\lambda_\eta^{(\ell_0-\ell_{\min})}, 
\end{align*}
where the last inequality comes from \eqref{h_H_0_bound}.

Following the same rationale, the norm of the third term can be bounded as
\begin{align}
    \left\|\mathbf{D}_3(\ell_0)\right\|
    &\leq 2 \sum_{q=1}^\infty \left(\left\|\bar{\mathbf{H}}_0^{-1}\right\|\left\|\mathbf{E}_0(\ell_0)\right\|_F\right)^q\left\|\bar{\mathbf{H}}_0^{-1}\mathbf{e}_0(\ell_0)\right\|\nonumber\\
    &\leq 2 \sum_{q=1}^\infty \left(\left\|\bar{\mathbf{H}}_0^{-1}\right\|\left\|\mathbf{E}_0(\ell_0)\right\|_F\right)^{q+1}
\end{align}
which leads to
\begin{align}\label{third_term}
    \left\|\mathbf{D}_3(\ell_0)\right\|
    &\leq 2\sum_{q=1}^\infty \left(\frac{1}{4}\tilde{\xi}\lambda_\eta^{(\ell_0-\ell_{\min})}\right)^{q+1}\\
    &=
    \frac{\tilde{\xi}\lambda_\eta^{(\ell_0-\ell_{\min})}}{\left(1-\frac{1}{4}\tilde{\xi}\lambda_\eta^{(\ell_0-\ell_{\min})}\right)}\cdot
    \frac{1}{8}\tilde{\xi}\lambda_\eta^{(\ell_0-\ell_{\min})}\\
    &<\frac{1}{2}\tilde{\xi}\lambda_\eta^{(\ell_0-\ell_{\min})}.\nonumber
\end{align}
where the last inequality has used the results in \eqref{trick}. Note that this bound is very loose since we bound a second order term with the first order term.

Substituting $\tilde{\xi}=IC\xi/(\sigma_{\min}^2D)$ defined in \eqref{tildexi} back to \eqref{first_term}, \eqref{second_term}, \eqref{third_term} and summing them up, we have
\begin{align*}
	\left\|\mathbf{d}_i^0(\ell_0)-\mathbf{d}_j^0(\ell_0)\right\|
	&\leq
	\left(1+\frac{\sigma_{\max} \epsilon_{\max} }{\sigma_{\min}^2}\right)\frac{IC}{\sigma_{\min}^2D}\xi\lambda_\eta^{(\ell_0-\ell_{\min})}
\end{align*}
Introducing the constants $C_1$ and $C_2$ defined in \eqref{C_infty} and the inequality in \eqref{upperbound_disagreement_0}, we have
\begin{align}
	\left\|\mathbf{x}_i^1-\mathbf{x}_j^1\right\|
    \leq
	\xi \left(\frac{CC_1C_2}{D}\right)\lambda_\eta^{(\ell_0-\ell_{\min})}.
\end{align}
and therefore the result holds for $k=1$.

\subsection{Induction: $k=K$ and $k=K+1$}

Let the error bound hold for $k=K$ such that for any $i\neq j$
\begin{align}\label{bound_x_K}
	\left\|\mathbf{x}_i^{K}-\mathbf{x}_j^{K}\right\|
	\leq
    \xi \left(\frac{CC_1C_2}{D}\right)\sum_{k=0}^K\lambda_\eta^{(\ell_k-\ell_{\min})}.
\end{align}
with $C_1,C_2$ given in \eqref{C_infty}. The inequality below holds
\begin{align*}
	\left\|\mathbf{x}_i^{K+1}-\mathbf{x}_j^{K+1}\right\|
	&\leq \left\|\mathbf{x}_i^{K}-\mathbf{x}_j^{K}\right\|
		   +\left\|\mathbf{d}_i^{K}(\ell_K)-\mathbf{d}_j^{K}(\ell_K)\right\|,
\end{align*}
where
\begin{align}\label{H_E_k}
	\mathbf{d}_i^{K}(\ell_K)-\mathbf{d}_j^{K}(\ell_K)
	&=
	\left[\bar{\mathbf{H}}_{K} +\mathbf{E}_{K,i}(\ell_K) \right]^{-1}\left[\bar{\mathbf{h}}_{K}  +  \mathbf{e}_{K,i}(\ell_K)\right]\nonumber\\
	&-\left[\bar{\mathbf{H}}_{K} +\mathbf{E}_{K,j}(\ell_K) \right]^{-1}\left[\bar{\mathbf{h}}_{K}  +  \mathbf{e}_{K,j}(\ell_K)\right].
\end{align}
Similar to the case when $k=1$, if the perturbations $\mathbf{E}_{K,i}(\ell_K)$, $\mathbf{E}_{K,j}(\ell_K)$ are small enough, the expansion in \eqref{expansion} can be applied here to simplify the expression.

\subsubsection{Matrix series expansion}
By definition \eqref{definition_delta_Delta}, we have
\begin{align}\label{H_k}
	\left\|\bar{\mathbf{H}}_{K}^{-1}\right\|
	=\left\|\left[\mathbf{Q}(\mathbf{x}_i^{K})+\bdsb{\Delta}_{K,i}\right]^{-1}\right\|,
\end{align}
which is another perturbed inverse. Thus we first examine whether this inverse can be expanded using the series expansion in \eqref{expansion}. From \eqref{lipschitz_delta_Delta} and \eqref{bound_x_K}, we have
\begin{align}\label{Delta_bound}
	\left\|\bdsb{\Delta}_{K,i}\right\| &< \nu_{\Delta}  \xi \left(\frac{CC_1C_2}{D}\right)\sum_{k=0}^K\lambda_\eta^{(\ell_k-\ell_{\min})}\\
    &< \xi\left(\frac{\nu_{\Delta} C C_1 C_2}{D}\right)\lambda_\infty,
\end{align}
where the last inequality comes from the non-negativity of $\lambda_\eta$ (i.e., $
\lambda_\infty>\sum_{k=0}^K\lambda_\eta^{(\ell_k-\ell_{\min})}$ for all finite $K$). By the definition of $D$ in \eqref{denominator} in Assumption \ref{ell_ass}, we have
\begin{align}\label{inequality_RDelta}
	\left\|\mathbf{Q}^{-1}(\mathbf{x}_i^{K})\bdsb{\Delta}_{K,i}\right\|
	&\leq \left\|\mathbf{Q}^{-1}(\mathbf{x}_i^{K})\right\|\left\|\bdsb{\Delta}_{K,i}\right\|\\
    &\leq \underbrace{\frac{I}{\sigma_{\min}^2} \left(\frac{\nu_{\Delta}  C C_1 C_2}{D}\right)\lambda_\infty}_{<1,~\textrm{from}~\eqref{denominator}} \xi < \xi
<1/2,\nonumber
\end{align}
where we have used the fact that $\left\|\mathbf{Q}^{-1}(\mathbf{x}_i^K)\right\|\leq I/\sigma_{\min}^2$ (see \eqref{h_H_0_bound}). Therefore, the matrix series expansion holds for \eqref{H_k}. Then using the above calculations, we have
\begin{align}\label{H_k_bound}
	\left\|\bar{\mathbf{H}}_{K}^{-1}\right\|
	&\leq \left\|\mathbf{Q}^{-1}(\mathbf{x}_i^{K})\right\|\\
	&~~~~ + \sum_{q=1}^\infty\left(\left\|\mathbf{Q}^{-1}(\mathbf{x}_i^{K})\right\|\left\|\bdsb{\Delta}_{K,i}\right\|\right)^q\left\|\mathbf{Q}^{-1}(\mathbf{x}_i^{K})\right\|\nonumber\\
	&\leq \frac{I}{\sigma_{\min}^2} +  \frac{I}{\sigma_{\min}^2}\sum_{q=1}^\infty\xi^q\nonumber\\
       &= \frac{I}{\sigma_{\min}^2}\left(1+\frac{\xi}{1-\xi}\right)=\frac{I}{\sigma_{\min}^2}\frac{1}{1-\xi}<\frac{2I}{\sigma_{\min}^2}.\nonumber
\end{align}
Similar to the case with $k=1$, we have
\begin{align*}
	\left\|\bar{\mathbf{H}}_{K}^{-1} \mathbf{E}_{K,i}(\ell_K)\right\| &\leq \left\|\bar{\mathbf{H}}_{K}^{-1}\right\|\left\|\mathbf{E}_{K,i}(\ell_K)\right\|\leq  \left\|\bar{\mathbf{H}}_{K}^{-1}\right\|\left\|\mathbf{E}_K(\ell_K)\right\|_F\\
	\left\|\bar{\mathbf{H}}_{K}^{-1} \mathbf{E}_{K,j}(\ell_K)\right\| &\leq \left\|\bar{\mathbf{H}}_{K}^{-1}\right\|\left\|\mathbf{E}_{K,j}(\ell_K)\right\|\leq  \left\|\bar{\mathbf{H}}_{K}^{-1}\right\|\left\|\mathbf{E}_K(\ell_K)\right\|_F.
\end{align*}
From Lemma \ref{lem_gossip_error} and Assumption \ref{ell_ass}, the above bound can be further bounded using \eqref{H_k_bound} as
\begin{align*}
    \left\|\bar{\mathbf{H}}_{K}^{-1}\right\|\left\|\mathbf{E}_K(\ell_K)\right\|_F
	&\leq \frac{2I}{\sigma_{\min}^2} C \lambda_\eta^{\ell_K}
    =\lambda_\eta^{(\ell_K-\ell_{\min})}\frac{IC}{2\sigma_{\min}^2 D}\xi.
\end{align*}
For notation convenience, we again let $\tilde{\xi} = {IC\xi}/(\sigma_{\min}^2 D)$ in \eqref{tildexi} with $\tilde{\xi}<\xi<1/2$ and let $\delta\mathbf{Z}=\mathbf{E}_{K,i}(\ell_K)$ or $\mathbf{E}_{K,j}(\ell_K)$ and $\mathbf{Z}=\bar{\mathbf{H}}_K$. As a result, we have
\begin{align}\label{ratio4}
    \left\|\mathbf{Z}^{-1}\delta\mathbf{Z}\right\|&=
    \left\|\bar{\mathbf{H}}_K^{-1}\right\|\left\|\mathbf{E}_K(\ell_K)\right\|_F\\
	&\leq \frac{1}{2}\lambda_\eta^{(\ell_K-\ell_{\min})} \tilde{\xi} < \frac{1}{4}.
\end{align}
Therefore, the matrix expansion holds. By grouping all the high order terms
$q\geq 1$ in the matrix expansion, we have
\begin{align*}
	&\mathbf{d}_i^K(\ell_K)-\mathbf{d}_j^K(\ell_K)\\
	&=
    \left[\bar{\mathbf{H}}_K^{-1}-\sum_{q=1}^\infty (-1)^q\left(\bar{\mathbf{H}}_K^{-1}\mathbf{E}_{K,i}(\ell_K)\right)^q\bar{\mathbf{H}}_K^{-1}\right]
	\left[\bar{\mathbf{h}}_K+\mathbf{e}_{K,i}(\ell_K)\right]\\
    & -\left[\bar{\mathbf{H}}_K^{-1}-\sum_{q=1}^\infty (-1)^q\left(\bar{\mathbf{H}}_K^{-1}\mathbf{E}_{K,j}(\ell_K)\right)^q\bar{\mathbf{H}}_K^{-1}\right]
	\left[\bar{\mathbf{h}}_K+\mathbf{e}_{K,j}(\ell_K)\right].
\end{align*}
To simplify the above expression, we write it in three terms $\mathbf{D}_1(\ell_K)$, $\mathbf{D}_2(\ell_K)$ and $\mathbf{D}_3(\ell_K)$ as follows
\begin{align}
    \mathbf{d}_i^K(\ell_K)-\mathbf{d}_j^K(\ell_K)
    = \mathbf{D}_1(\ell_K) + \mathbf{D}_2(\ell_K) + \mathbf{D}_3(\ell_K),
\end{align}
where $\mathbf{D}_1(\ell_K) \triangleq \bar{\mathbf{H}}_K^{-1}\left[\mathbf{e}_{K,i}(\ell_K)-\mathbf{e}_{K,j}(\ell_K)\right]$ and
\begin{align*}
	\mathbf{D}_2(\ell_K)
    &\triangleq
    \sum_{q=1}^\infty(-1)^q\left(\bar{\mathbf{H}}_K^{-1}\mathbf{E}_{K,j}(\ell_K)\right)^q\bar{\mathbf{H}}_K^{-1}\bar{\mathbf{h}}_K\\
    &~~~-\sum_{q=1}^\infty(-1)^q\left(\bar{\mathbf{H}}_K^{-1}\mathbf{E}_{K,i}(\ell_K)\right)^q\bar{\mathbf{H}}_K^{-1}\bar{\mathbf{h}}_K\nonumber\\
    \mathbf{D}_3(\ell_K) &\triangleq \sum_{q=1}^\infty(-1)^q\left(\bar{\mathbf{H}}_K^{-1}\mathbf{E}_{K,j}(\ell_K)\right)^q\bar{\mathbf{H}}_K^{-1}\mathbf{e}_{K,j}(\ell_K)\\
    &~~~~~ -\sum_{q=1}^\infty(-1)^q\left(\bar{\mathbf{H}}_K^{-1}\mathbf{E}_{K,i}(\ell_K)\right)^q\bar{\mathbf{H}}_K^{-1}\mathbf{e}_{K,i}(\ell_K).
\end{align*}

\subsubsection{Proof of success when $k=K+1$}
According to the triangular inequality for norms, we can bound
\begin{align*}
	\left\|\mathbf{e}_{K,i}(\ell_K) - \mathbf{e}_{K,j}(\ell_K)\right\| &\leq
2\left\|\mathbf{e}_K(\ell_K)\right\|\\
	\left\|\mathbf{E}_{K,i}(\ell_K) - \mathbf{E}_{K,j}(\ell_K)\right\| &\leq 2\left\|\mathbf{E}_K(\ell_K)\right\|_F
\end{align*}

Using \eqref{ratio4}, we can bound the norm of the first term as
\begin{align}\label{first_term_K}
    \left\|\mathbf{D}_1(\ell_K)\right\|
    &\leq 2\left\|\bar{\mathbf{H}}_K^{-1}\right\|\left\|\mathbf{e}_K(\ell_K)\right\|
    \leq \tilde{\xi}\lambda_\eta^{(\ell_K-\ell_{\min})}.
\end{align}
Similarly, the infinite sum in the second term is bounded as
\begin{align}\label{second_term_K}
    &\left\|\sum_{q=1}^\infty(-1)^q\left[\left(\bar{\mathbf{H}}_K^{-1}\mathbf{E}_{K,j}(\ell_K)\right)^q-\left(\bar{\mathbf{H}}_K^{-1}\mathbf{E}_{K,i}(\ell_K)\right)^q\right]\right\|\nonumber\\
    &\leq
    2\sum_{q=1}^\infty \left(\left\|\bar{\mathbf{H}}_K^{-1}\right\|\left\|\mathbf{E}_K(\ell_K)\right\|_F\right)^q\\
    &\leq 2\sum_{q=1}^\infty \left(\frac{1}{2}\tilde{\xi}\lambda_\eta^{(\ell_K-\ell_{\min})}\right)^q
    =
    \frac{\tilde{\xi}\lambda_\eta^{(\ell_K-\ell_{\min})}}{\left(1-\frac{1}{2}\tilde{\xi}\lambda_\eta^{(\ell_K-\ell_{\min})}\right)},\nonumber
\end{align}
where the last equality comes from the convergence of geometric series $\lim_{K\rightarrow\infty}\sum_{k=1}^{K}a^k=a/(1-a)$ for any $|a|<1$. Since $0<\tilde{\xi}<\xi<1/2$ and $\lambda_\eta^{(\ell_K-\ell_{\min})}<1$, then
\begin{align}\label{trick_K}
    \frac{\tilde{\xi}\lambda_\eta^{(\ell_K-\ell_{\min})}}{\left(1-\frac{1}{2}\tilde{\xi}\lambda_\eta^{(\ell_K-\ell_{\min})}\right)}<2\tilde{\xi}\lambda_\eta^{(\ell_K-\ell_{\min})}
\end{align}
and thus the norm of the second term is bounded as
\begin{align}
    \left\|\mathbf{D}_2(\ell_K)\right\|
    &\leq 2\tilde{\xi}\lambda_\eta^{(\ell_K-\ell_{\min})} \left\|\bar{\mathbf{H}}_K^{-1}\right\|\left\|\bar{\mathbf{h}}_K\right\|\\
    &\leq  \frac{2\sigma_{\max} \epsilon_{\max} }{\sigma_{\min}^2}\tilde{\xi}\lambda_\eta^{(\ell_K-\ell_{\min})}, 
\end{align}
where the last inequality comes from \eqref{h_H_0_bound}. Following the same rationale, the norm of the third term can be bounded as
\begin{align}\label{third_term_K}
    \left\|\mathbf{D}_3(\ell_K)\right\|
    &\leq 2 \sum_{q=1}^\infty \left(\left\|\bar{\mathbf{H}}_K^{-1}\right\|\left\|\mathbf{E}_K(\ell_K)\right\|_F\right)^q\left\|\bar{\mathbf{H}}_K^{-1}\mathbf{e}_K(\ell_K)\right\|\nonumber\\
    &\leq 2 \sum_{q=1}^\infty \left(\left\|\bar{\mathbf{H}}_K^{-1}\right\|\left\|\mathbf{E}_K(\ell_K)\right\|_F\right)^{q+1}\nonumber\\
    &\leq 2\sum_{q=1}^\infty \left(\frac{1}{2}\tilde{\xi}\lambda_\eta^{(\ell_K-\ell_{\min})}\right)^{q+1}\nonumber\\
    &=
    \frac{\tilde{\xi}\lambda_\eta^{(\ell_K-\ell_{\min})}}{\left(1-\frac{1}{2}\tilde{\xi}\lambda_\eta^{(\ell_K-\ell_{\min})}\right)}\cdot
    \frac{1}{2}\tilde{\xi}\lambda_\eta^{(\ell_K-\ell_{\min})}
    <\tilde{\xi}\lambda_\eta^{(\ell_K-\ell_{\min})}.\nonumber
\end{align}
where the last inequality has used the results in \eqref{trick_K}. Note that this is again a very loose bound.

Substituting $\tilde{\xi}=IC\xi/(\sigma_{\min}^2D)$ in \eqref{tildexi} back to \eqref{first_term_K}, \eqref{second_term_K} and \eqref{third_term_K} and using the constants $C_1$ and $C_2$, we have
\begin{align*}
	\left\|\mathbf{d}_i^K(\ell_K)-\mathbf{d}_j^K(\ell_K)\right\|
	&\leq
\xi \left(\frac{CC_1C_2}{D}\right) \lambda_\eta^{(\ell_K-\ell_{\min})}
\end{align*}
Similarly, based on \eqref{C_infty} and \eqref{upperbound_disagreement_0}, we have
\begin{align*}
	&\left\|\mathbf{x}_i^{K+1}-\mathbf{x}_j^{K+1}\right\|\\
    &\leq
    \left\|\mathbf{x}_i^K-\mathbf{x}_j^K\right\|+\left\|\mathbf{d}_i^K(\ell_K)-\mathbf{d}_j^K(\ell_K)\right\|\\
    &\leq
	  \xi \left(\frac{CC_1C_2}{D}\right) \sum_{k=0}^{K-1}\lambda_\eta^{(\ell_k-\ell_{\min})}
      + \xi \left(\frac{CC_1C_2}{D}\right)\lambda_\eta^{(\ell_K-\ell_{\min})}\\
    &= \xi \left(\frac{CC_1C_2}{D}\right) \sum_{k=0}^K\lambda_\eta^{(\ell_k-\ell_{\min})},
\end{align*}
and therefore given that the recursion holds for $k=K$, it holds true for $k=K+1$. The induction is complete. Given \eqref{denominator}, we have $\xi \leq 4D \lambda_\eta^{(\ell_{\min}+1)}$, and
\begin{align}
	\left\|\mathbf{x}_i^{K+1}-\mathbf{x}_j^{K+1}\right\|
    &\leq
    4 CC_1C_2\sum_{k=0}^K\lambda_\eta^{\ell_k+1}.
\end{align}

\section{Proof of Theorem \ref{proposition_discrepancy}}\label{proof_jroposition_discrepancy}
By the decomposition in \eqref{h_H_decomp}, we have
\begin{align}\label{discrepancy}
	&\mathbf{d}_i^k(\ell_k) - \mathbf{d}_i^k\\
	&=\left[\mathbf{Q}(\mathbf{x}_i^k)+\bdsb{\Delta}_{k,i} +\mathbf{E}_{k,i}(\ell_k) \right]^{-1}
	\left[\mathbf{q}(\mathbf{x}_i^k) +\bdsb{\delta}_{k,i}  +  \mathbf{e}_{k,i}(\ell_k)\right] \nonumber\\
	&~~~ - \mathbf{Q}(\mathbf{x}_i^k)^{-1}\mathbf{q}(\mathbf{x}_i^k).\nonumber
\end{align}
Now that we verify that the matrix series expansion holds for similar approximations. First of all, from Lemma \ref{lem_disagreement} and in particular \eqref{Delta_bound}, we have $\left\|\bdsb{\Delta}_{k,i}(\ell_k)\right\|\leq \nu_{\Delta}C_1\lambda_\infty\xi/D$. The expansion depends on the quantity
\begin{align*}
	&\left\|\mathbf{Q}^{-1}(\mathbf{x}_i^k)\left(\bdsb{\Delta}_{k,i}+\mathbf{E}_{k,i}(\ell_k)\right)\right\|\\
	&\leq \left\|\mathbf{Q}^{-1}(\mathbf{x}_i^k)\right\|\left\|\bdsb{\Delta}_{k,i}\right\| + \left\|\mathbf{Q}^{-1}(\mathbf{x}_i^k)\right\|\left\|\mathbf{E}_{K,i}(\ell_k)\right\|.
\end{align*}
Using the derivation in \eqref{inequality_RDelta} and $C_1,C_2$ in \eqref{C_infty}, we have
\begin{align}\label{R_bound}
	&\left\|\mathbf{Q}^{-1}(\mathbf{x}_i^k) \left(\bdsb{\Delta}_{k,i}+\mathbf{E}_{k,i}(\ell_k)\right)\right\|\nonumber\\
	&<C_2\left(\frac{\nu_{\Delta}  C C_1 C_2}{D}\right)\lambda_\infty\xi + \left(\frac{CC_2}{4D}\right)\xi\lambda_\eta^{(\ell_k-\ell_0)}\nonumber\\
    &= \frac{CC_2}{D} \left(\nu_\Delta C_1C_2 \lambda_\infty+\frac{1}{4}\lambda_\eta^{(\ell_k-\ell_0)}\right)\xi\nonumber\\
    &<\underbrace{\frac{CC_2\left(\nu\lambda_\infty C_1C_2+1\right)}{D}}_{=1,~\textrm{from}~\eqref{denominator}}\xi<\xi<\frac{1}{2},
\end{align}
where the last inequality is by the definition of $D$ in \eqref{denominator}. Then \eqref{discrepancy} can be re-written as
\begin{align*}
	&\mathbf{d}_i^k(\ell_k) - \mathbf{d}_i^k\\
	&=
	\left[\mathbf{Q}^{-1}(\mathbf{x}_i^k)-\sum_{q=1}^\infty\left(\mathbf{Q}^{-1}(\mathbf{x}_i^k)(\bdsb{\Delta}_{k,i} +\mathbf{E}_{k,i}(\ell_k))\right)^q\mathbf{Q}^{-1}(\mathbf{x}_i^k)\right]\\
    &~~~~ \times\left[\mathbf{q}(\mathbf{x}_i^k) +\bdsb{\delta}_{k,i}  +  \mathbf{e}_{k,i}(\ell_k)\right] - \mathbf{Q}^{-1}(\mathbf{x}_i^k)\mathbf{q}(\mathbf{x}_i^k)\\
    &=\mathbf{Q}^{-1}(\mathbf{x}_i^k)\left[\bdsb{\delta}_{k,i}  +  \mathbf{e}_{k,i}(\ell_k)\right]\\
    &~~~~ -\sum_{q=1}^\infty\left(\mathbf{Q}^{-1}(\mathbf{x}_i^k)(\bdsb{\Delta}_{k,i} +\mathbf{E}_{k,i}(\ell_k))\right)^q\mathbf{Q}^{-1}(\mathbf{x}_i^k)\mathbf{q}(\mathbf{x}_i^k)\\
    &~~~~ -\sum_{q=1}^\infty\left(\mathbf{Q}^{-1}(\mathbf{x}_i^k)(\bdsb{\Delta}_{k,i} +\mathbf{E}_{k,i}(\ell_k))\right)^q\mathbf{Q}^{-1}(\mathbf{x}_i^k)\left[\bdsb{\delta}_{k,i}  +  \mathbf{e}_{k,i}(\ell_k)\right]
\end{align*}
According to Lemma \ref{lem_disagreement} and Assumption \ref{ell_ass}, we have $\left\|\bdsb{\delta}_{k,i}(\ell_k)\right\|\leq \nu_{\delta} C C_1 C_2\lambda_\infty\xi/D$, and the norm of the first term above can be bounded similarly as \eqref{R_bound}
\begin{align}
    \left\|\mathbf{Q}^{-1}(\mathbf{x}_i^k)\left[\bdsb{\delta}_{k,i}  +  \mathbf{e}_{k,i}(\ell_k)\right]\right\|
    &<\xi.
\end{align}
Likewise, the norm of the second term is bounded as
\begin{align*}
    &\left\|\sum_{q=1}^\infty\left(\mathbf{Q}^{-1}(\mathbf{x}_i^k)(\bdsb{\Delta}_{k,i} +\mathbf{E}_{k,i}(\ell_k))\right)^q\mathbf{Q}^{-1}(\mathbf{x}_i^k)\mathbf{q}(\mathbf{x}_i^k)\right\|\\
    &\leq \sum_{q=1}^\infty\left(\left\|\mathbf{Q}^{-1}(\mathbf{x}_i^k)(\bdsb{\Delta}_{k,i} +\mathbf{E}_{k,i}(\ell_k))\right\|\right)^q\left\|\mathbf{Q}^{-1}(\mathbf{x}_i^k)\mathbf{q}(\mathbf{x}_i^k)\right\|\\
    &<\frac{\sigma_{\max} \epsilon_{\max} }{\sigma_{\min}^2}\sum_{q=1}^\infty \xi^q = \frac{\sigma_{\max} \epsilon_{\max} }{\sigma_{\min}^2} \frac{\xi}{1-\xi}
    <2\frac{\sigma_{\max} \epsilon_{\max} }{\sigma_{\min}^2}\xi,
\end{align*}
and similarly for the third term
\begin{align*}
    &\left\|\sum_{q=1}^\infty\left(\mathbf{Q}^{-1}(\mathbf{x}_i^k)(\bdsb{\Delta}_{k,i} +\mathbf{E}_{k,i}(\ell_k))\right)^q\mathbf{Q}^{-1}(\mathbf{x}_i^k)\left[\bdsb{\delta}_{k,i}  +  \mathbf{e}_{k,i}(\ell_k)\right]\right\|\\
    &<\sum_{q=1}^\infty \left\|\mathbf{Q}^{-1}(\mathbf{x}_i^k)(\bdsb{\Delta}_{k,i} +\mathbf{E}_{k,i}(\ell_k))\right\|^{q+1}
    <\sum_{q=1}^\infty\xi^{q+1} = \frac{\xi^2}{1-\xi}.
\end{align*}
Furthermore, since $\xi\in(0,1/2)$, the above expression can be simplified as $\xi^2/(1-\xi)<2\xi$. Finally, summing them up and using the constant $C_1$ we have
\begin{align}
	\left\|\mathbf{d}_i^k(\ell_k) - \mathbf{d}_i^k\right\|\leq 2\left(1+\frac{\sigma_{\max} \epsilon_{\max} }{\sigma_{\min}^2}\right)\xi=C_1\xi
\end{align}
for all $i$ and $k$. Now we have established that the discrepancy between the decentralized descent and the exact descent can be bounded by an arbitrarily small error $\xi$ specified by the system. Given \eqref{denominator}, we have
\begin{align}
    \xi < 4D \lambda_\eta^{(\ell_{\min}+1)},
\end{align}
and therefore, the perturbation bound $\kappa$ on the error recursion in Lemma \ref{lem_error_recursion} can be obtained as
\begin{align}
    \kappa \triangleq 4C_1 D \lambda_\eta^{(\ell_{\min}+1)}.
\end{align}

\bibliographystyle{IEEEtran}


\bibliography{../ref_general,../ref_power_system_SE,../ref_dist_opt}

\begin{thebibliography}{10}
\providecommand{\url}[1]{#1}
\csname url@samestyle\endcsname
\providecommand{\newblock}{\relax}
\providecommand{\bibinfo}[2]{#2}
\providecommand{\BIBentrySTDinterwordspacing}{\spaceskip=0pt\relax}
\providecommand{\BIBentryALTinterwordstretchfactor}{4}
\providecommand{\BIBentryALTinterwordspacing}{\spaceskip=\fontdimen2\font plus
\BIBentryALTinterwordstretchfactor\fontdimen3\font minus
  \fontdimen4\font\relax}
\providecommand{\BIBforeignlanguage}[2]{{%
\expandafter\ifx\csname l@#1\endcsname\relax
\typeout{** WARNING: IEEEtran.bst: No hyphenation pattern has been}%
\typeout{** loaded for the language `#1'. Using the pattern for}%
\typeout{** the default language instead.}%
\else
\language=\csname l@#1\endcsname
\fi
#2}}
\providecommand{\BIBdecl}{\relax}
\BIBdecl

\bibitem{nocedal1999numerical}
J.~Nocedal and S.~Wright, \emph{{N}umerical {O}ptimization}.\hskip 1em plus
  0.5em minus 0.4em\relax Springer verlag, 1999.

\bibitem{dennis1996numerical}
J.~Dennis and R.~Schnabel, \emph{{N}umerical {M}ethods for {U}nconstrained
  {O}ptimization and {N}onlinear {E}quations}.\hskip 1em plus 0.5em minus
  0.4em\relax Society for Industrial Mathematics, 1996, vol.~16.

\bibitem{boyd2004convex}
S.~Boyd and L.~Vandenberghe, \emph{{C}onvex {O}ptimization}.\hskip 1em plus
  0.5em minus 0.4em\relax Cambridge Univ Pr, 2004.

\bibitem{bjorck1996numerical}
{\AA}.~Bj{\"o}rck, \emph{{N}umerical {M}ethods for {L}east {S}quares
  {P}roblems}.\hskip 1em plus 0.5em minus 0.4em\relax Society for Industrial
  Mathematics, 1996, no.~51.

\bibitem{monticelli2000electric}
A.~Monticelli, ``{E}lectric {P}ower {S}ystem {S}tate {E}stimation,''
  \emph{Proceedings of the IEEE}, vol.~88, no.~2, pp. 262--282, 2000.

\bibitem{mensing2006positioning}
C.~Mensing and S.~Plass, ``{P}ositioning {A}lgorithms for {C}ellular {N}etworks
  using {TDOA},'' in \emph{Acoustics, Speech and Signal Processing, 2006.
  ICASSP 2006 Proceedings. 2006 IEEE International Conference on},
  vol.~4.\hskip 1em plus 0.5em minus 0.4em\relax IEEE, 2006, pp. IV--IV.

\bibitem{stoica1989maximum}
P.~Stoica, R.~Moses, B.~Friedlander, and T.~Soderstrom, ``{M}aximum
  {L}ikelihood {E}stimation of the {P}arameters of {M}ultiple {S}inusoids from
  {N}oisy {M}easurements,'' \emph{Acoustics, Speech and IEEE Trans. Signal
  Process.}, vol.~37, no.~3, pp. 378--392, 1989.

\bibitem{bell1993iterated}
B.~Bell and F.~Cathey, ``The {I}terated {K}alman {F}ilter {U}pdate as a
  {G}auss-{N}ewton {M}ethod,'' \emph{Automatic Control, IEEE Transactions on},
  vol.~38, no.~2, pp. 294--297, 1993.

\bibitem{schweiger2005gauss}
M.~Schweiger, S.~Arridge, and I.~Nissil{\"a}, ``{G}auss-{N}ewton method for
  {I}mage {R}econstruction in {D}iffuse {O}ptical {T}omography,'' \emph{Physics
  in medicine and biology}, vol.~50, p. 2365, 2005.

\bibitem{tsitsiklis1984problems}
J.~Tsitsiklis, ``{P}roblems in {D}ecentralized {D}ecision {M}aking and
  {C}omputation.'' DTIC Document, Tech. Rep., 1984.

\bibitem{karp2000randomized}
R.~Karp, C.~Schindelhauer, S.~Shenker, and B.~Vocking, ``{R}andomized {R}umor
  {S}preading,'' in \emph{Foundations of Computer Science, 2000. Proceedings.
  41st Annual Symposium on}.\hskip 1em plus 0.5em minus 0.4em\relax IEEE, 2000,
  pp. 565--574.

\bibitem{olfati2004consensus}
R.~Olfati-Saber and R.~Murray, ``{C}onsensus {P}roblems in {N}etworks of
  {A}gents with {S}witching {T}opology and {T}ime-{D}elays,'' \emph{Automatic
  Control, IEEE Transactions on}, vol.~49, no.~9, pp. 1520--1533, 2004.

\bibitem{dimakis2010gossip}
A.~Dimakis, S.~Kar, J.~Moura, M.~Rabbat, and A.~Scaglione, ``{G}ossip
  {A}lgorithms for {D}istributed {S}ignal {P}rocessing,'' \emph{Proc. IEEE},
  vol.~98, no.~11, pp. 1847--1864, 2010.

\bibitem{kempe2003gossip}
D.~Kempe, A.~Dobra, and J.~Gehrke, ``{G}ossip-based {C}omputation of
  {A}ggregate {I}nformation,'' in \emph{Foundations of Computer Science, 2003.
  Proceedings. 44th Annual IEEE Symposium on}.\hskip 1em plus 0.5em minus
  0.4em\relax IEEE, 2003, pp. 482--491.

\bibitem{boyd2006randomized}
S.~Boyd, A.~Ghosh, B.~Prabhakar, and D.~Shah, ``{R}andomized {G}ossip
  {A}lgorithms,'' \emph{IEEE Trans. Inf. Theory}, vol.~52, no.~6, pp.
  2508--2530, 2006.

\bibitem{olfati2007consensus}
R.~Olfati-Saber, J.~Fax, and R.~Murray, ``{C}onsensus and {C}ooperation in
  {N}etworked {M}ulti-agent {S}ystems,'' \emph{Proceedings of the IEEE},
  vol.~95, no.~1, pp. 215--233, 2007.

\bibitem{johansson2008subgradient}
B.~Johansson, T.~Keviczky, M.~Johansson, and K.~Johansson, ``{S}ubgradient
  {M}ethods and {C}onsensus {A}lgorithms for {S}olving {C}onvex {O}ptimization
  {P}roblems,'' in \emph{Decision and Control, 2008. CDC 2008. 47th IEEE
  Conference on}.\hskip 1em plus 0.5em minus 0.4em\relax IEEE, 2008, pp.
  4185--4190.

\bibitem{bertsekas1997new}
D.~Bertsekas, M.~I. of~Technology. Laboratory~for Information, and D.~Systems,
  ``{A} {N}ew {C}lass of {I}ncremental {G}radient {M}ethods for {L}east
  {S}quares {P}roblems,'' \emph{SIAM Journal on Optimization}, vol.~7, no.~4,
  pp. 913--926, 1997.

\bibitem{nedic2001incremental}
A.~Nedic and D.~Bertsekas, ``{I}ncremental {S}ubgradient {M}ethods for
  {N}on-differentiable {O}ptimization,'' \emph{SIAM Journal of Optimization},
  vol.~12, no.~1, pp. 109--138, 2001.

\bibitem{nedic2009distributed}
A.~Nedic and A.~Ozdaglar, ``{D}istributed {S}ubgradient {M}ethods for
  {M}ulti-agent {O}ptimization,'' \emph{Automatic Control, IEEE Transactions
  on}, vol.~54, no.~1, pp. 48--61, 2009.

\bibitem{ram2010distributed}
S.~Ram, A.~Nedic, and V.~Veeravalli, ``{D}istributed {S}tochastic {S}ubgradient
  {P}rojection {A}lgorithms for {C}onvex {O}ptimization,'' \emph{Journal of
  optimization theory and applications}, vol. 147, no.~3, pp. 516--545, 2010.

\bibitem{nedich2010asynchronous}
A.~Nedic, ``{A}synchronous {B}roadcast-based {C}onvex {O}ptimization over a
  {N}etwork,'' \emph{Automatic Control, IEEE Transactions on}, no.~99, pp.
  1--1, 2010.

\bibitem{srivastava2011distributed}
K.~Srivastava and A.~Nedic, ``{D}istributed {A}synchronous {C}onstrained
  {S}tochastic {O}ptimization,'' \emph{Selected Topics in Signal Processing,
  IEEE Journal of}, no.~99, pp. 1--1, 2011.

\bibitem{kar2008distributed}
S.~Kar, J.~Moura, and K.~Ramanan, ``{Distributed Parameter Estimation in Sensor
  Networks: Nonlinear Observation Models and Imperfect Communication},''
  \emph{Information Theory, IEEE Transactions on}, vol.~58, no.~6, pp. 3575
  --3605, june 2012.

\bibitem{matei2011performance}
I.~Matei and J.~Baras, ``{P}erformance {E}valuation of the {C}onsensus-based
  {D}istributed {S}ubgradient {M}ethod under {R}andom {C}ommunication
  {T}opologies,'' \emph{Selected Topics in Signal Processing, IEEE Journal of},
  no.~99, pp. 1--1, 2011.

\bibitem{chen2011diffusion}
J.~Chen and A.~Sayed, ``{D}iffusion {A}daptation {S}trategies for {D}istributed
  {O}ptimization and {L}earning over {N}etworks,'' \emph{Signal Processing,
  IEEE Transactions on}, vol.~60, no.~8, pp. 4289--4305, 2012.

\bibitem{lopes2008diffusion}
C.~Lopes and A.~Sayed, ``{D}iffusion {L}east-{M}ean {S}quares over {A}daptive
  {N}etworks: {F}ormulation and {P}erformance {A}nalysis,'' \emph{IEEE Trans.
  Signal Process.}, vol.~56, no.~7, pp. 3122--3136, 2008.

\bibitem{cattivelli2010diffusion}
F.~Cattivelli and A.~Sayed, ``{D}iffusion {LMS} {S}trategies for {D}istributed
  {E}stimation,'' \emph{IEEE Trans. Signal Process.}, vol.~58, no.~3, pp.
  1035--1048, 2010.

\bibitem{cattivelli2008diffusion}
F.~Cattivelli, C.~Lopes, and A.~Sayed, ``{D}iffusion {R}ecursive
  {L}east-{S}quares for {D}istributed {E}stimation over {A}daptive
  {N}etworks,'' \emph{IEEE Trans. Signal Process.}, vol.~56, no.~5, pp.
  1865--1877, 2008.

\bibitem{wei2010distributed}
E.~Wei, A.~Ozdaglar, and A.~Jadbabaie, ``{A} {D}istributed {N}ewton {M}ethod
  for {N}etwork {U}tility {M}aximization,'' in \emph{Decision and Control
  (CDC), 2010 49th IEEE Conference on}.\hskip 1em plus 0.5em minus 0.4em\relax
  IEEE, 2010, pp. 1816--1821.

\bibitem{ilic2012toward}
M.~Ilic and A.~Hsu, ``{T}oward {D}istributed {C}ontingency {S}creening using
  {L}ine {F}low {C}alculators and {D}ynamic {L}ine {R}ating {U}nits ({DLR}s),''
  in \emph{2012 45th Hawaii International Conference on System Sciences}.\hskip
  1em plus 0.5em minus 0.4em\relax IEEE, 2012, pp. 2027--2035.

\bibitem{bejar2010distributed}
B.~Bejar, P.~Belanovic, and S.~Zazo, ``{D}istributed {G}auss-{N}ewton {M}ethod
  for {L}ocalization in {A}d-hoc {N}etworks,'' in \emph{Signals, Systems and
  Computers (ASILOMAR), 2010 Conference Record of the Forty Fourth Asilomar
  Conference on}.\hskip 1em plus 0.5em minus 0.4em\relax IEEE, 2010, pp.
  1452--1454.

\bibitem{cheng2005distributed}
B.~Cheng, R.~Hudson, F.~Lorenzelli, L.~Vandenberghe, and K.~Yao,
  ``{D}istributed {G}auss-{N}ewton {M}ethod for {N}ode {L}ocalization in
  {W}ireless {S}ensor {N}etworks,'' in \emph{Signal Processing Advances in
  Wireless Communications, 2005 IEEE 6th Workshop on}.\hskip 1em plus 0.5em
  minus 0.4em\relax IEEE, 2005, pp. 915--919.

\bibitem{calafiore2010distributed}
G.~Calafiore, L.~Carlone, and M.~Wei, ``{A} {D}istributed {G}auss-{N}ewton
  {A}pproach for {R}ange-based {L}ocalization of {M}ulti-agent {F}ormations,''
  in \emph{Computer-Aided Control System Design (CACSD), 2010 IEEE
  International Symposium on}.\hskip 1em plus 0.5em minus 0.4em\relax IEEE,
  2010, pp. 1152--1157.

\bibitem{zhao2007information}
T.~Zhao and A.~Nehorai, ``{I}nformation-{D}riven {D}istributed {M}aximum
  {L}ikelihood {E}stimation based on {G}auss-{N}ewton {M}ethod in {W}ireless
  {S}ensor {N}etworks,'' \emph{IEEE Trans. Signal Process.}, vol.~55, no.~9,
  pp. 4669--4682, 2007.

\bibitem{schweppe1974static}
F.~Schweppe and E.~Handschin, ``{S}tatic {S}tate {E}stimation in {E}lectric
  {P}ower {S}ystems,'' \emph{Proceedings of the IEEE}, vol.~62, no.~7, pp.
  972--982, 1974.

\bibitem{larson1970state}
R.~Larson, W.~Tinney, and J.~Peschon, ``{S}tate {E}stimation in {P}ower
  {S}ystems {P}art {I}: {T}heory and {F}easibility,'' \emph{IEEE Trans. Power
  App. Syst.}, no. 3Part-I, pp. 345--352, 1970.

\bibitem{brice1982multiprocessor}
C.~Brice and R.~Cavin, ``{M}ultiprocessor {S}tatic {S}tate {E}stimation,''
  \emph{IEEE Trans. Power App. Syst.}, no.~2, pp. 302--308, 1982.

\bibitem{kurzyn1983real}
M.~Kurzyn, ``{R}eal-{T}ime {S}tate {E}stimation for {L}arge-{S}cale {P}ower
  {S}ystems,'' \emph{IEEE Trans. Power App. Syst.}, no.~7, pp. 2055--2063,
  1983.

\bibitem{yang2011transition}
T.~Yang, H.~Sun, and A.~Bose, ``{T}ransition to a {T}wo-{L}evel {L}inear
  {S}tate {E}stimator : {P}art i \& ii,'' \emph{IEEE Trans. Power Syst.},
  no.~99, pp. 1--1, 2011.

\bibitem{gomez2011multilevel}
A.~G{\'o}mez-Exp{\'o}sito, A.~Abur, A.~de~la Villa~Ja{\'e}n, and
  C.~G{\'o}mez-Quiles, ``{A} {M}ultilevel {S}tate {E}stimation {P}aradigm for
  {S}mart {G}rids,'' \emph{Proceedings of the IEEE}, no.~99, pp. 1--25, 2011.

\bibitem{falcao1995parallel}
D.~Falcao, F.~Wu, and L.~Murphy, ``{P}arallel and {D}istributed {S}tate
  {E}stimation,'' \emph{IEEE Trans. Power Syst.}, vol.~10, no.~2, pp. 724--730,
  1995.

\bibitem{lin1992distributed}
S.~Lin, ``{A} {D}istributed {S}tate {E}stimator for {E}lectric {P}ower
  {S}ystems,'' \emph{IEEE Trans. Power Syst.}, vol.~7, no.~2, pp. 551--557,
  1992.

\bibitem{ebrahimian2000state}
R.~Ebrahimian and R.~Baldick, ``{S}tate {E}stimation {D}istributed
  {P}rocessing,'' \emph{IEEE Trans. Power Syst.}, vol.~15, no.~4, pp.
  1240--1246, 2000.

\bibitem{van1981two}
T.~Van~Cutsem, J.~Horward, and M.~Ribbens-Pavella, ``{A} {T}wo-{L}evel {S}tatic
  {S}tate {E}stimator for {E}lectric {P}ower {S}ystems,'' \emph{IEEE Trans.
  Power App. Syst.}, no.~8, pp. 3722--3732, 1981.

\bibitem{zhao2005multi}
L.~Zhao and A.~Abur, ``{M}ulti-area {S}tate {E}stimation using {S}ynchronized
  {P}hasor {M}easurements,'' \emph{IEEE Trans. Power Syst.}, vol.~20, no.~2,
  pp. 611--617, 2005.

\bibitem{jiang2007distributed}
W.~Jiang, V.~Vittal, and G.~Heydt, ``{A} {D}istributed {S}tate {E}stimator
  {U}tilizing {S}ynchronized {P}hasor {M}easurements,'' \emph{IEEE Trans. Power
  Syst.}, vol.~22, no.~2, pp. 563--571, 2007.

\bibitem{xie2012fully}
L.~Xie, D.~Choi, S.~Kar, and H.~Poor, ``{F}ully {D}istributed {S}tate
  {E}stimation for {W}ide-{A}rea {M}onitoring {S}ystems,'' \emph{Smart Grid,
  IEEE Transactions on}, vol.~3, no.~3, pp. 1154--1169, 2012.

\bibitem{kekatos2012distributed}
V.~Kekatos and G.~Giannakis, ``{D}istributed {R}obust {P}ower {S}ystem {S}tate
  {E}stimation,'' \emph{Arxiv preprint arXiv:1204.0991}, 2012.

\bibitem{li2013optimal}
X.~Li, A.~Scaglione, and T.-H. Chang, ``{Optimal Sensor Placement for Hybrid
  State Estimation in Smart Grid},'' \emph{Acoustics, Speech and Signal
  Processing, 2013. ICASSP 2013 Proceedings. 2013 IEEE International Conference
  on}.

\bibitem{horntopics}
R.~Horn and C.~Johnson, ``{T}opics in {M}atrix {A}nalysis, 1991.''

\bibitem{eriksson2004applied}
K.~Eriksson, D.~Estep, and C.~Johnson, \emph{{A}pplied {M}athematics, {B}ody
  and {S}oul: {D}erivates and {G}eometry in $\mathbb{R}^3$}.\hskip 1em plus
  0.5em minus 0.4em\relax Springer Verlag, 2004, vol.~3.

\bibitem{blondel2005convergence}
V.~Blondel, J.~Hendrickx, A.~Olshevsky, and J.~Tsitsiklis, ``{C}onvergence in
  {M}ultiagent {C}oordination, {C}onsensus, and {F}locking,'' in \emph{Decision
  and Control, 2005 and 2005 European Control Conference. CDC-ECC'05. 44th IEEE
  Conference on}.\hskip 1em plus 0.5em minus 0.4em\relax IEEE, 2005, pp.
  2996--3000.

\bibitem{monticelli1999state}
A.~Monticelli, \emph{{State Estimation in Electric Power Systems: a Generalized
  Approach}}.\hskip 1em plus 0.5em minus 0.4em\relax Springer, 1999, vol. 507.

\bibitem{UK_grid}
\BIBentryALTinterwordspacing
``{U}. {K}. {N}ational {G}rid-{R}eal {T}ime {O}perational {D}ata,'' 2009,
  [Online; accessed 22-July-2004]. [Online]. Available:
  \url{http://www.nationalgrid.com/uk/Electricity/Data/}
\BIBentrySTDinterwordspacing

\bibitem{salzo2011convergence}
S.~Salzo and S.~Villa, ``{C}onvergence {A}nalysis of a {P}roximal
  {G}auss-{N}ewton {M}ethod,'' \emph{Arxiv preprint arXiv:1103.0414}, 2011.

\bibitem{galor2007discrete}
O.~Galor, \emph{{Discrete Dynamical Systems}}.\hskip 1em plus 0.5em minus
  0.4em\relax Springer, 2007.

\end{thebibliography}

\end{document}